\documentclass[reqno,11pt]{amsart}

\usepackage{amsmath,amsfonts,amssymb,amsthm,epsfig,amscd,}
\usepackage[]{fontenc}
\usepackage[latin1]{inputenc}
\usepackage{enumerate}
\usepackage[cmtip,all]{xy}
\usepackage{tabularx,multirow}
\usepackage[usenames,dvipsnames]{color}

 \allowdisplaybreaks \numberwithin{equation}{section}

\newtheorem{theorem}{Theorem}[section]
\newtheorem{proposition}[theorem]{Proposition}
\newtheorem{lemma}[theorem]{Lemma}

\newtheorem{claim}[theorem]{Claim}

\theoremstyle{definition}
\newtheorem{definition}[theorem]{Definition}
\newtheorem{example}[theorem]{Example}

\newtheorem{remark}[theorem]{Remark}

\newcommand{\ord}{\mathrm{ord}}
\newcommand{\Aut}{\mathrm{Aut}}
\newcommand{\mult}{\mathrm{mult}}
\newcommand{\codim}{\mathrm{codim}}
\newcommand{\Sing}{\mathrm{Sing}}
\newcommand{\Hilb}{\mathrm{Hilb}}
\newcommand{\Supp}{\mathrm{Supp}}

\definecolor{red}{rgb}{0.8, 0.2, 0.0}

\begin{document}

\title[On large theta-characteristics with prescribed vanishing]{On large theta-characteristics with prescribed vanishing}
\author{Edoardo Ballico}
\address{Edoardo Ballico, Dipartimento di Matematica, Universit\`a degli Studi di Trento, Via Sommarive 14, 38123 Povo (TN) - Italy}
\email{edoardo.ballico@unitn.it}
\author{Francesco Bastianelli}
\address{Francesco Bastianelli, Dipartimento di Matematica e Fisica, Universit\`a degli Studi Roma Tre, Largo San Leonardo Murialdo 1, 00146 Roma - Italy}
\email{bastiane@mat.uniroma3.it}
\thanks{This work was partially supported by FIRB 2012 \emph{``Spazi di moduli e applicazioni''}; PRIN 2010-11 \emph{``Geometry of algebraic varieties''}; FAR 2010 (Milano-Bicocca) \emph{``Geometria e Topologia''}; INdAM (GNSAGA)}
\author{Luca Benzo}
\address{Luca Benzo, Dipartimento di Scienze Matematiche "Giuseppe Luigi Lagrange", Politecnico di Torino, Corso Duca degli Abruzzi 24, 10129 Torino - Italy}
\email{luca.benzo@polito.it}

\begin{abstract}
Let $C$ be a smooth projective curve of genus $g\geq 2$.
Fix an integer $r\geq 0$, and let $\underline{k}=(k_1,\ldots,k_n)$ be a sequence of positive integers with $\sum_{i=1}^n k_i =g-1$.
In this paper, we study $n$-pointed curves $(C,p_1,\ldots,p_n)$ such that the line bundle $L:=O_C\left(\sum_{i=1}^n k_i p_i\right)$ is a theta-characteristic with $h^0\left(C,L\right)\geq r+1$ and $h^0\left(C,L\right)\equiv r+1 \,\mathrm{(mod.\,2)}$.
We prove that they describe a sublocus $\mathcal{G}^r_g(\underline{k})$ of $\mathcal{M}_{g,n}$ having codimension at most $g-1+\frac{r(r-1)}{2}$.
Moreover, for any $r\geq 0$, $\underline{k}$ as above, and $g$ greater than an explicit integer $g(r)$ depending on $r$, we present irreducible components of $\mathcal{G}^r_g(\underline{k})$ attaining the maximal codimension in $\mathcal{M}_{g,n}$, so that the bound turns out to be sharp.
\end{abstract}

\maketitle
\section{Introduction}\label{section INTRODUCTION}

The study of theta-characteristics on algebraic curves has been developed along various directions concerning, for instance, the dimension of spaces of global sections, or the geometry of effective divisors in the associated linear series.
Consider a smooth complex projective curve $C$ of genus $g\geq 2$, and let $\mathcal{M}_g$ be the moduli space of smooth curves of genus $g$.
We recall that a theta-characteristic $L$ is a line bundle on $C$ such that $L^{\otimes 2}\cong \omega_C$, and its \emph{parity} is the residue modulo 2 of the dimension $h^0\left(C, L\right)$ of the space of global sections.

In the seminal paper \cite{Mu}, Mumford introduced a purely algebraic approach to theta-characteristics, and he proved that the parity is invariant under flat deformations of pairs $(C,L)$.
Along these lines, Harris \cite{Ha} focused on the loci $\mathcal{M}_g^r$ in $\mathcal{M}_g$ of curves admitting a \emph{large} theta-characteristic $L$, that is $h^0(C,L)\geq r+1$ and $h^0(C,L)\equiv r+1 (\text{mod } 2)$ for some fixed integer $r\geq 0$.
In particular, he showed that the codimension of each irreducible component of $\mathcal{M}_g^r$ is at most $\frac{r(r+1)}{2}$.
The geometry of these loci has been thoroughly investigated (see e.g. \cite{Te, Na, Col, Fa1, Fa2}), and recently, Harris' bound has been proved to be sharp for any value of $r\geq 0$ and $g\geq g(r)$, where $g(r)$ is an integer depending on $r$ (cf. \cite{Be1, Be2}).

On the other hand, Kontsevich and Zorich  \cite{KZ} set their analysis of theta-characteristics in the moduli space of abelian differentials, which parameterizes isomorphism classes of pairs $(C,\omega)$ consisting of a smooth curve $C$ of genus $g$ endowed with a non-zero holomorphic form $\omega\in H^0(C,\omega_C)$.
Given any partition $\underline{k}=(k_1,\ldots,k_n)$ of the integer $g-1$---i.e. a sequence of integers $k_1\geq\dots\geq k_n>0$ such that $\sum_{i=1}^n k_i=g-1$---they obtain relevant results on the number of connected components of the locus $\mathcal{H}_g(2\underline{k})$ described by pairs $(C,\omega)$ such that $\omega$ vanishes along a divisor of the form $2(k_1p_1+\dots+k_n p_n)$, so that $\mathcal{O}_C\left(\sum_{i=1}^n k_ip_i\right)$ is a theta-characteristic on $C$.

\smallskip
In this paper we consider both the viewpoints above at once, and we focus on the subvarieties of the moduli space $\mathcal{M}_{g,n}$ described by $n$-pointed curves of genus $g$ admitting a large theta-characteristic with a global section vanishing with prescribed multiplicities at the marked points.
More precisely, given an integer $r\geq 0$ and a partition $\underline{k}=(k_1,\ldots,k_n)$ of $g-1$, we are aimed at studying the loci defined as
\begin{equation*}\label{equation G_g^r(k)}
\mathcal{G}_{g}^r(\underline{k}):=\left\{[C,p_1,\ldots,p_n]\in \mathcal{M}_{g,n}\left|
\begin{array}{l}
\displaystyle L:=\mathcal{O}_C\left(\sum_{i=1}^n k_ip_i\right)\textrm{ is a theta-characteristic having}\\ h^0\left(C,L\right)\geq r+1
\textrm{ and } h^0\left(C,L\right)\equiv r+1\,(\mathrm{mod}\,2)
\end{array}\right.\right\}.
\end{equation*}
In particular, we achieve a general upper bound governing the codimension of $\mathcal{G}_{g}^r(\underline{k})$ in $\mathcal{M}_{g,n}$.
Furthermore, for any $r\geq 0$ and $g$ greater than an explicit integer $g(r)$ depending on $r$, we present irreducible components of $\mathcal{G}_{g}^r(\underline{k})$ having maximal codimension, so that the bound turns out to be sharp.

Clearly, any locus $\mathcal{G}_{g}^r(\underline{k})$ maps on $\mathcal{M}_g^r$ under the forgetful morphism $\pi_n\colon \mathcal{M}_{g,n}\longrightarrow \mathcal{M}_g$.
In particular, when $\underline{k}=(1,\ldots,1)$, $\mathcal{G}_{g}^r(\underline{k})$ dominates $\mathcal{M}_g^r$, and  $\dim \mathcal{G}_{g}^r(\underline{k})\geq \dim \mathcal{M}^r_g +r$ as the fibre over a general $[C]\in \mathcal{M}_g^r$ is described by the complete linear series $|L|$ associated to the large theta-characteristic on $C$.
On the other hand, if $\underline{k}=(g-1)$, then the subvarieties $\mathcal{G}_{g}^r:=\mathcal{G}_{g}^r(g-1)$ are the \emph{loci of subcanonical points}, which recently gained renewed interest (see e.g. \cite{BP, Bu, CT}).
Finally, it is worth noticing that the description of these subvarieties---or of corresponding loci in different moduli spaces---led to various applications in classical enumerative and projective geometry \cite{Ha}, in differential geometry \cite{Pi, BP}, and in dynamical systems \cite{KZ}.

\smallskip
By means of Harris' bound, we prove the following.
\begin{theorem}\label{theorem BOUND}
Fix integers $r\geq 0$, $g>n>0$, and consider a partition $\underline{k}=(k_1,\ldots,k_n)$ of $g-1$. Then either $\mathcal{G}_{g}^r(\underline{k})$ is empty, or the codimension in ${\mathcal{M}_{g,n}}$  of each irreducible component $\mathcal{Z}$ of $\mathcal{G}_{g}^r(\underline{k})$ satisfies
\begin{equation*}
\codim_{\mathcal{M}_{g,n}} \mathcal{Z}\leq g-1+\frac{r(r-1)}{2}.
\end{equation*}
\end{theorem}

Accordingly, we shall say that an irreducible component $\mathcal{Z}\subset \mathcal{G}_{g}^r(\underline{k})$ has \emph{expected dimension} if it satisfies equality in the latter bound, that is $\dim \mathcal{Z}= 2g-2+n-\frac{r(r-1)}{2}$.
We note that when $\underline{k}=(1,\ldots,1)$, Theorem \ref{theorem BOUND} agrees with the bound on the codimension of $\mathcal{M}_g^r$, as it gives $\dim \mathcal{Z}\geq \left(3g-3-\frac{r(r+1)}{2}\right)+r$.
Moreover, the assertion for $\underline{k}=(g-1)$ coincides with \cite[Theorem 1.1]{BP}, and the proof of our result relies on a similar argument.

\smallskip
For any $r\geq 0$, we then consider the integer $g(r)$ defined by
\begin{equation}\label{equation g(r)}
\displaystyle g(r):=\left\{
\begin{array}{ll}
2 & \textrm{for }r=0\\
3r & \textrm{for }1\leq r\leq 3\\
{r+2\choose 2} & \textrm{for }r\geq 4.
\end{array}\right.
\end{equation}
We prove the sharpness of the bound in Theorem \ref{theorem BOUND} for any $r\geq 0$ and $g\geq g(r)$.
Namely,
\begin{theorem}\label{theorem EXP DIM}
For any genus $g\geq g(r)$, and for any partition $\underline{k}=(k_1,\ldots,k_n)$ of $g-1$, the locus $\mathcal{G}_g^{r}(\underline{k})$ is non-empty, and there exists an irreducible component $\mathcal{Z}\subset\mathcal{G}_{g}^r(\underline{k})$ having expected dimension.
In particular, at a general point $[C,p_1,\ldots,p_n]\in \mathcal{Z}$, the large theta-characteristic $\mathcal{O}_C\left(\sum_{i=1}^n k_ip_i\right)$ possesses exactly $r+1$ independent global sections and, apart from the cases $(r,g)=(0,2)$ and $(1,3)$, the curve $C$ is non-hyperelliptic.
\end{theorem}
In the light of the relation between $\mathcal{M}^r_g$ and $\mathcal{G}_{g}^r(\underline{k})$, the assertion for $\underline{k}=(1,\ldots,1)$ is well-known for small values of $r$ (cf. \cite{Te}), and it is included in \cite[Theorem 1.2]{Be1} for arbitrary $r$.
In particular, the value of $g(r)$ can be lowered in this case (see \cite{Fa1, Be2}).
Besides, the statement in the case $\underline{k}=(g-1)$ with $0\leq r\leq 3$ is covered by different results in \cite{KZ, Bu, BP}, so Theorem \ref{theorem EXP DIM} extends them to arbitrary large values of $r$.

We note further that under the assumption $\underline{k}=(g-1)$, the bound in Theorem \ref{theorem BOUND} is meaningful as long as the expected dimension is non-negative, that is $ g\geq \left\lfloor\frac{r^2-r+4}{4}\right\rfloor$, which is hypothetically the best value for $g(r)$ when $r$ is large enough.

\smallskip
In order to prove Theorem \ref{theorem EXP DIM}, we firstly show that the whole statement follows from the assertion for $\underline{k}=(g-1)$ (see Theorem \ref{theorem PARTITION}).
Roughly speaking, this depends on the fact that each irreducible component of the subcanonical locus $\mathcal{G}_{g}^r$ may be thought as a degeneration of any $\mathcal{G}_{g}^r(\underline{k})$ (cf. also \cite[Corollary 2]{KZ}).
Therefore the assertion for small values of $r$ follows from known results on the loci of subcanonical points.

As far as arbitrary large values of $r$ are concerned, we argue by induction on $g$ and $r$ separately, in analogy with \cite{Fa1, Be1} where the sharpness of Harris' bound is proved.
Thanks to \cite[Theorem 4.1]{BP}, if $g(r)$ is an integer such that $\mathcal{G}_{g(r)}^r$ admits an irreducible component $\mathcal{Z}_{g(r)}$ having expected dimension, then for any $g\geq g(r)$, the locus
$\mathcal{G}_{g}^r$ has an irreducible component with the same property.
By using basic facts on Eisenbud-Harris' limit linear series \cite{EH1, EH2}, we slightly improve the latter result in order to keep track of global sections (see Theorem \ref{theorem SMOOTHING g}).

In the spirit of \cite{Be1}, we then set $g(r):={r+2 \choose 2}$, and we prove that if $\mathcal{Z}_{g(r-1)}\subset \mathcal{G}_{g(r-1)}^{r-1}$ is an irreducible component satisfying equality in Theorem \ref{theorem BOUND}, then there exists an irreducible component $\mathcal{Z}_{g(r)}\subset \mathcal{G}_{g(r)}^r$ which still has expected dimension (cf. Theorem \ref{theorem SMOOTHING r}), so that Theorem \ref{theorem EXP DIM} follows.
The proof of such a result is mainly based on deformation theory and Ran's description of Hilbert schemes of points on nodal curves \cite{Ra1, Ra2, Ra3, Ra4, Ra5}.
In particular, we consider nodal reducible curves in $\mathbb{P}^r$ consisting of an elliptic normal curve suitably attached to a degenerate curve $C$ such that $[C]\in \mathcal{M}^{r-1}_{g(r-1)}$.
It follows from \cite{Be1}, that these singular curves can be deformed to smooth curves $X\subset \mathbb{P}^r$ such that $[X]\in \mathcal{M}^{r}_{g(r)}$, where $\mathcal{O}_X(1)$ is the corresponding large theta-characteristic.
Therefore a pair $[X,p]$ lies on $\mathcal{G}_{g(r)}^r$, i.e. $p\in X$ is a \emph{subcanonical point} of $X$, if and only if the divisor $\left(g(r)-1\right)p$ is cut out on $X$ by some hyperplane.
So we extend in these terms the notion of being a subcanonical point on the nodal curves.
Finally, we show that the nodal curves can be smoothed into a component of expected dimension, preserving the property of having a subcanonical point.

\smallskip
The paper is organized as follows.
In Section \ref{section LARGE THETA} we prove Theorem \ref{theorem BOUND}, we reduce the proof of Theorem \ref{theorem EXP DIM} to the case $\underline{k}=(g-1)$, and we recall some preliminary results about loci of subcanonical points.
Then Section \ref{section EXP DIM} is entirely devoted to perform the induction on $r$, and to conclude the proof of Theorem \ref{theorem EXP DIM}.

\section{Loci of large theta-characteristics with prescribed vanishing}\label{section LARGE THETA}

In this section we firstly aim at proving Theorem \ref{theorem BOUND}.
Then we shall reduce the proof of Theorem \ref{theorem EXP DIM} to the case of loci of subcanonical points, and we shall recall some preliminary results on this topic.

\subsection{Notation}

We work throughout over the field $\mathbb{C}$ of complex numbers.
By \emph{curve} we mean a complete connected reduced algebraic curve over $\mathbb{C}$.
Given a variety $X$, we say that a property holds for a \emph{general} point $x\in X$ if it holds on a non-empty open subset of X.

As is customary, we denote by $\mathcal{M}_g$ the moduli space of smooth curves of genus $g$. Let $\mathcal{M}_{g,n}$ be the moduli space consisting of isomorphism classes of ordered $(n+1)$-tuples $(C,p_1,\ldots,p_n)$, where $[C]\in \mathcal{M}_g$ and $p_1,\ldots,p_n\in C$ are distinct points, and let $\pi_n\colon \mathcal{M}_{g,n}\longrightarrow \mathcal{M}_{g}$ be the \emph{forgetful} morphism sending $[C,p_1,\ldots,p_n]$ to $[C]$.

Moreover, let $\mathcal{S}_g$ denote the moduli space of spin curves, which parameterizes pairs $[C,L]$ such that $[C]\in \mathcal{M}_g$ and $L$ is a theta-characteristic on $C$.
Hence there is a natural map $\varphi\colon \mathcal{S}_g\longrightarrow \mathcal{M}_g$ of degree $2^{2g}$ sending $[C,L]$ to $[C]$.
According to the notation on $\mathcal{M}_g$, we define $\mathcal{S}^r_g:=\left\{[C,L]\in \mathcal{S}_g\left|h^0\left(C,L\right)\geq r+1\textrm{ and } h^0\left(C,L\right)\equiv r+1\,(\mathrm{mod}\,2)\right.\right\}$.
In addition, we consider the moduli space $\mathcal{S}_{g,n}$ of $(n+2)$-tuples $[C,L,p_1,\ldots,p_n]$ with $[C,p_1,\ldots,p_n]\in \mathcal{M}_{g,n}$ and $[C,L]\in \mathcal{S}_g$, and we denote by $\mathcal{S}^r_{g,n}$ the sublocus of $n$-pointed spin curves such that $[C,L]\in \mathcal{S}^r_g$.

\subsection{Bound on the codimension of $\mathcal{G}_{g}^r(\underline{k})$ in $\mathcal{M}_{g,n}$}

We are now going to prove Theorem \ref{theorem BOUND}.
As in the proof of  \cite[Theorem 1.1]{BP}, the result is achieved by intersecting cycles on the relative symmetric product of families of curves with large theta-characteristics.

\begin{proof}[Proof of Theorem \ref{theorem BOUND}]
Let us fix integers $r\geq 0$, $g>n>0$, and consider a partition $\underline{k}=(k_1,\ldots,k_n)$ of $g-1$.
Assuming that $\mathcal{G}_{g}^r(\underline{k})$ is non-empty, we consider a point ${[C,p_1,\ldots,p_n]\in \mathcal{G}_{g}^r(\underline{k})}$.
Hence we want to prove that any irreducible component $\mathcal{Z}\subset \mathcal{G}_{g}^r(\underline{k})$ passing through $[C,p_1,\ldots,p_n]$ has dimension ${\dim \mathcal{Z} \geq  2g-2+n-\frac{r(r-1)}{2}}$.\\
By assumption, the line bundle $L:=\mathcal{O}_C\left(\sum_{i=1}^n k_ip_i\right)$ is a theta-characteristic having $h^0\left(C,L\right)\geq r+1$ and $h^0\left(C,L\right)\equiv r+1\,(\mathrm{mod}\,2)$, so that ${[C,L,p_1,\ldots,p_n]\in \mathcal{S}_{g,n}^r}$.
Therefore we can consider a versal deformation family ${\left(\mathcal{C}\stackrel{\phi}{\longrightarrow} U, \mathcal{L}\longrightarrow\mathcal{C},  U\stackrel{\rho_i}{\longrightarrow} \mathcal{C}\right)}$ of the $n$-pointed curve $(C,L,p_1,\ldots,p_n)$ in $\mathcal{S}_{g,n}$.
So there is a commutative diagram of finite maps
\begin{equation}\label{equation VERSAL DEFORMATION}
\xymatrix{ U \ar[r] \ar[d] & U/ \Aut(C,L,p_1,\ldots,p_n) \ar@{^{(}->}[r] & \mathcal{S}_{g,n}\ar[d]^{\varphi_n} \\  U' \ar[r] & U'/ \Aut(C,p_1,\ldots,p_n) \ar@{^{(}->}[r] & \mathcal{M}_{g,n} }
\end{equation}
where $U'$ is a versal deformation space of $(C,p_1,\ldots,p_n)$ in $\mathcal{M}_{g,n}$, and $\varphi_n$ is the natural map of degree $2^{2g}$.
In particular, the family ${\phi\colon \mathcal{C}\longrightarrow U}$ consists of smooth curves $C_t:=\phi^{-1}(t)$ of genus $g$, the line bundle $\mathcal{L}\longrightarrow\mathcal{C}$ restricts to a theta-characteristics ${L_t:=\mathcal{L}_{|C_t}}$ on each fibre, and for $i=1,\ldots,n$, the maps ${\rho_i\colon U\longrightarrow \mathcal{C}}$ are sections of $\phi$ with ${p_{i,t}:=\rho_i(t)\in C_t}$.
Moreover, we may assume $(C_{0},L_{0},p_{1,0},\ldots,p_{n,0})=(C,L,p_1,\ldots,p_n)$ for some point $0\in U$.\\
Then we restrict the versal deformation to the locus ${U^r:=\left\{t\in U\left|[C_t,L_t,p_{1,t},\ldots,p_{n,t}]\in \mathcal{S}_{g,n}^r\right.\right\}}$, and we consider the $(g-1)$-fold relative symmetric product ${\mathcal{C}^{(g-1)}\stackrel{\Phi}{\longrightarrow} U^r}$ of the family $\mathcal{C}$, so that the fibre over each $t$ is the $(g-1)$-fold symmetric product $C_t^{(g-1)}$ of the curve $C_t$.
Let us define two subvarieties of $\mathcal{C}^{(g-1)}$ as
$$
{\mathcal{P}:=\left\{\left.k_1p_{1,t}+\dots+k_np_{n,t}\in C_t^{(g-1)}\right|t\in U^r\right\}},
$$
which restricts to a point of the $\underline{k}$-diagonal on each fibre $C_t^{(g-1)}$, and
$$
{\mathcal{Y}:=\left\{\left. q_1+\dots+q_{g-1}\in C_t^{(g-1)}\right|t\in U^r \textrm{ and }\mathcal{O}_{C_t}\left(q_1+\dots+q_{g-1}\right)\cong L_t\right\}},
$$
which parameterizes effective divisors $Q_t\in C_t^{(g-1)}$ in the linear systems $|L_t|$.
Thus, if ${k_1p_{1,t}+\dots+k_np_{n,t}\in\mathcal{P}\cap\mathcal{Y}}$, the line bundle $\mathcal{O}_{C_t}\left(\sum_{i=1}^n k_ip_{i,t}\right)$ is a theta-characteristic such that ${h^0\left(C_t,\mathcal{O}_{C_t}\left(\sum_{i=1}^n k_ip_{i,t}\right)\right)\geq r+1}$ and ${h^0\left(C_t,\mathcal{O}_{C_t}\left(\sum_{i=1}^n k_ip_{i,t}\right)\right)\equiv r+1\,(\mathrm{mod}\,2)}$, that is ${[C_t,p_{1,t},\ldots,p_{n,t}]\in\mathcal{G}_{g}^r(\underline{k})}$.
It follows from the diagram (\ref{equation VERSAL DEFORMATION}), that the map ${U\longrightarrow \mathcal{M}_{g,n}}$ given by ${t\longmapsto [C_t,p_{1,t}\ldots,p_{n,t}]}$ is finite.
Therefore each irreducible component $\mathcal{Z}\subset \mathcal{G}_{g}^r(\underline{k})$ passing through $[C,p_{1},\ldots,p_{n}]$ satisfies
\begin{equation}\label{equation DIMENSION}
\dim \mathcal{Z} \geq \dim \mathcal{P} + \dim \mathcal{Y} - \dim \mathcal{C}^{(g-1)},
\end{equation}
which is the minimal dimension of any irreducible component of $\mathcal{P}\cap \mathcal{Y}$.
We point out that ${\dim \mathcal{P}=\dim U^r\geq 3g-3+n-\frac{r(r+1)}{2}}$ by Harris' bound (cf. \cite[Corollary 1.11]{Ha}).
Moreover,  $\dim \mathcal{Y}\geq\dim U^r + r$ and ${\dim \mathcal{C}^{(g-1)}=\dim U^r + g-1}$.
Thus ${\dim \mathcal{Z} \geq  2g-2+n-\frac{r(r-1)}{2}}$, and the assertion follows.
\end{proof}

We note that when the integer $g$ is small, the combinatorics of the loci $\mathcal{G}_{g}^r(\underline{k})$ is very simple.
In particular, these loci can be easily described, and they always have expected dimension, as the following example shows (cf. \cite[Theorem 2]{KZ} and \cite[Section 4]{Bu}).

\begin{example}[Low genera]\label{example LOW GENERA}	
In the case $g=2$, the unique locus to consider is $\mathcal{G}_{2}^0(1)$.
It is the irreducible divisor of $\mathcal{M}_{2,1}$ parameterizing pairs $[C,p]$, where $p\in C$ is a ramification point of the hyperelliptic map of $C$, and we have $h^0\left(C,\mathcal{O}_C(p)\right)=1$.\\
When $g=3$, each locus $\mathcal{G}_{3}^r(\underline{k})$ is still irreducible, with $0\leq r\leq 1$ and $\underline{k}\in \{(1,1),(2)\}$.
Recall that the canonical model of a non-hyperelliptic curve is a quartic curve in $C\subset \mathbb{P}^2$, and theta-characteristics are cut out by bitangent lines.
Therefore $\mathcal{G}_{3}^0(1,1)$ is the $6$-dimensional locus described by triples $(C,p_1,p_2)$ such that $p_1$ and $p_2$ have the same tangent line, whereas $\mathcal{G}_{3}^0(2)$ has dimension $5$ and parameterizes plane quartics endowed with a 4-inflection point.
On the other hand, hyperelliptic curves of genus $3$ give rise to $\mathcal{G}_{3}^1(1,1)$ and $\mathcal{G}_{3}^1(2)$. In particular, the former locus has dimension $6$ and consists of triples $[C,p_1,p_2]$, where $p_1$ and $p_2$ are conjugated under the hyperelliptic involution, whereas $\mathcal{G}_{3}^1(2)$ is the $5$-dimensional variety described by pairs $[C,p]$, where $p\in C$ is a Weierstrass point.
Furthermore, it is easy to see that for any point $[C,p_1,p_2]\in \mathcal{G}_{3}^r(1,1)$ (resp. $[C,p]\in\mathcal{G}_{3}^r(2)$) as above, we have $h^0\left(C,\mathcal{O}_C(p_1+p_2)\right)=r+1$ (resp. $h^0\left(C,\mathcal{O}_C(p)\right)=r+1$).\\
We note finally that if $p_1,p_2\in C$ are distinct Weierstrass points of a hyperelliptic curve $C$ of genus 3, then $[C,p_1,p_2]\in \mathcal{G}_{3}^0(1,1)$, and these triples span a divisor of $\mathcal{G}_{3}^0(1,1)$.
\end{example}

We point out that the locus $\mathcal{M}_g^r$ is empty if and only if $r>\frac{g-1}{2}$ (see e.g. \cite{Te}).
This fact follows from Clifford's Theorem and from the description of linear series on hyperelliptic curves (cf. \cite{ACGH}).
In particular, for any $0\leq r\leq \left\lfloor\frac{g-1}{2}\right\rfloor$, hyperelliptic curves of genus $g$ possess theta-characteristics with exactly $r+1$ global sections.
The example below shows that hyperelliptic curves provide several examples of large theta-characteristics with prescribed vanishing.
However, the loci they describe may not have expected dimension, and they do not cover all the possibilities for $r$ and $\underline{k}$.
This is indeed one of the main reasons for focusing on non-hyperelliptic curves.

\begin{example}[Hyperelliptic locus]\label{example HYPERELLIPTIC}
Let $L$ be a theta-characteristic on a hyperelliptic curve $C$ of genus $g$.
Effective divisors in the linear series $|L|$ have the form $\sum \alpha_i w_i + \sum \beta_j(p_j+q_j)$, where each $w_i$ is a different Weierstass point, the distinct pairs $(p_j,q_j)\in C^2$ consist of points conjugated by the hyperelliptic involution, and the positive integers $\alpha_i,\beta_j$ are such that $\sum \alpha_i+2\sum \beta_j=g-1$.
In particular, the dimension of $|L|$ is $r=\sum \lfloor \frac{\alpha_i}{2}\rfloor+ \sum \beta_j$ (cf. \cite[p. 13]{ACGH} and \cite[Appendix B]{KZ}).\\
Then, if $\underline{k}$ is the partition of $g-1$ obtained by reordering $(\alpha_1,\ldots ,\alpha_s,\beta_1,\beta_1,\ldots,\beta_t,\beta_t)$, the corresponding tuples of the form $(C, w_1,\ldots ,w_s,p_1,q_1,\ldots, p_t,q_t )$ describe a sublocus of $\mathcal{G}^r_g(\underline{k})$ of dimension $2g-1+t$.
As it has been observed at the end of Example \ref{example LOW GENERA}---under the assuption $g=3$, $s=2$ and $t=0$---such a sublocus may not be a whole component of  $\mathcal{G}^r_g(\underline{k})$.\\
Furthermore, if for instance $\underline{k}=(g-1)$, hyperelliptic curves describe an irreducible component of $\mathcal{G}^r_g(g-1)$ consisting of pairs $[C,w_1]$, so that its dimension is $2g-1$ and $r=\lfloor \frac{g-1}{2}\rfloor$.
In particular, choosing $\rho\in \left\{0,1\right\}$ with the same parity as $r+1$, such a component may be viewed as an irreducible component of $\mathcal{G}^\rho_g(g-1)$ having expected dimension.
However, no components with different parity can contain hyperelliptic curves.
\end{example}

\subsection{Reduction to the case of loci of subcanonical points}

In this subsection, we argue as in Theorem \ref{theorem BOUND} and we show that, in order to prove Theorem \ref{theorem EXP DIM}, it is enough to focus on the partition $\underline{k}=(g-1)$.

We recall that the set of partitions $\underline{k}$ of $g-1$ is naturally endowed with the structure of partially ordered set, as follows.
Given two partitions $\underline{k}=(k_1,\ldots,k_n)$ and $\underline{h}=(h_1,\ldots,h_m)$, we have that $\underline{k}\preceq\underline{h}$ if and only if there exists a partition $I_1\cup\dots\cup I_m$ of the set of indices $I=\{1,\ldots,n\}$ such that $h_j:=\sum_{i\in I_j}k_i$ for any $j=1,\ldots,m$.
For instance, we have that $\underline{k}\preceq (g-1)$ and $(1,\ldots,1)\preceq\underline{k}$ for any partition $\underline{k}$ of $g-1$: the former relation is obtained by taking the partition $I_1=I$, the latter one by partitioning the set $\{1,\ldots,g-1\}$ as $\{1,\ldots,k_1\}\cup\{k_1+1,\ldots,k_1+k_2\}\cup\dots \cup\{k_{n-1}+1,\ldots,g-1\}$.

We prove the following.
\begin{theorem}\label{theorem PARTITION}
Let $\underline{h}=(h_1,\ldots,h_m)$ be a partition of $g-1$, and assume that there exists an irreducible component $\mathcal{Z}\subset \mathcal{G}^r_g(\underline{h})$ having expected dimension.
Then for any partition $\underline{k}=(k_1,\ldots,k_n)$ satisfying $\underline{k}\preceq\underline{h}$, there exists an irreducible component $\mathcal{W}\subset \mathcal{G}^r_g(\underline{k})$ having expected dimension.\\
Furthermore, if the general point $[C,p_1,\ldots,p_m]\in \mathcal{Z}$ consists of a non-hyperelliptic curve $C$ such that $h^0(C, \mathcal{O}_C(\sum_{j=1}^m h_j p_j))=r+1$, then the general point $[D,q_1,\ldots,q_n]\in \mathcal{W}$ parameterizes a non-hyperelliptic curve $D$ such that $h^0\left(D, \mathcal{O}_{D}\left(\sum_{i=1}^n k_i q_i\right)\right)=r+1$, as well.\\
In particular, if the subcanonical locus $\mathcal{G}^r_g(g-1)$ admits an irreducible component of expected dimension with general point as above, then $\mathcal{G}^r_g(\underline{k})$ does for any partition $\underline{k}$ of $g-1$.
\begin{proof}
We notice that for any partition $\underline{k}$ satisfying $\underline{k}\preceq\underline{h}$, there exists a chain of relations $\underline{k}\preceq\underline{k}^{n-1}\preceq\dots\preceq\underline{k}^{m+1}\preceq\underline{h}$ such that for any $n-1\geq l\geq m+1$, the sequence $\underline{k}^{l}$ consists of exactly $l$ integers.
Thus it suffices to prove the statement in the case $n=m+1$, so that the whole assertion follows by iteration.
Accordingly, we assume hereafter that $n=m+1$.\\
Consider a point $[C, p_1,\ldots,p_m]\in\mathcal{Z}\subset \mathcal{G}^r_g(\underline{h})$, and let $(C,L)$ be the associated spin curve, where ${L:=\mathcal{O}_C\left(\sum_{j=1}^m h_jp_j\right)}$.
We argue as in the proof of Theorem \ref{theorem BOUND}, and we consider a versal deformation family ${\left(\mathcal{C}\stackrel{\phi}{\longrightarrow} U, \mathcal{L}\longrightarrow\mathcal{C}\right)}$ of $(C,L)$ in the moduli space $\mathcal{S}_{g}$ of spin curves.
We set $C_t:=\phi^{-1}(t)$ and ${L_t:=\mathcal{L}_{|C_t}}$, where $[C_0,L_0]=[C,L]$ for some $0\in U$.
Let ${\mathcal{C}^{(g-1)}\stackrel{\Phi}{\longrightarrow} U^r}$ be the restriction of the relative $(g-1)$-fold symmetric product to the locus ${U^r:=\left\{t\in U\left|[C_t,L_t]\in \mathcal{S}_{g}^r\right.\right\}}$, and set
${\mathcal{Y}:=\left\{\left. q_1+\dots+q_{g-1}\in C_t^{(g-1)}\right|t\in U^r \textrm{ and }\mathcal{O}_{C_t}\left(q_1+\dots+q_{g-1}\right)\cong L_t\right\}}$.
Moreover, we denote by ${\Delta_{\underline{k}}:=\left\{\left. k_1q_1+\dots+k_nq_{n}\in C_t^{(g-1)}\right|t\in U^r\right\}}$ the relative $\underline{k}$-diagonal.\\
Since ${\dim \Delta_{\underline{k}}=\dim U^r+n\geq 3g-3+n-\frac{r(r+1)}{2}}$ by \cite[Corollary 1.11]{Ha},  $\dim \mathcal{Y}\geq\dim U^r + r$, and ${\dim \mathcal{C}^{(g-1)}=\dim U^r + g-1}$, the dimension of any irreducible component  $\mathcal{W}_{\underline{k}}$ of $\mathcal{Y}\cap \Delta_{\underline{k}}$ is bounded by
\begin{equation}\label{equation DIMENSION2}
\dim \mathcal{W}_{\underline{k}}\geq \dim \mathcal{Y} + \dim \Delta_{\underline{k}} - \dim \mathcal{C}^{(g-1)}\geq 2g-2+n-\frac{r(r-1)}{2}.
\end{equation}
We recall that $\underline{k}\preceq \underline{h}$ and $n=m+1$.
Therefore the relative $\underline{h}$-diagonal $\Delta_{\underline{h}}$ is a codimension 1 subvariety of $\Delta_{\underline{k}}$, and the analogue of (\ref{equation DIMENSION2}) gives $\dim \left(\mathcal{Y}\cap \Delta_{\underline{h}}\right) \geq 2g-2+m-\frac{r(r-1)}{2}$, with $m=n-1$.
We point out that, if $h_1q_1+\dots+h_mq_m\in \mathcal{Y}\cap \Delta_{\underline{h}}$ and the $q_j$'s are distinct, the isomorphism class $[C_t, q_1,\ldots,q_m]\in \mathcal{M}_{g,m}$ lies in $\mathcal{G}^r_m(\underline{h})$.
Moreover, any component of $\mathcal{Y}\cap \Delta_{\underline{h}}$ passing through  $h_1q_1+\dots+h_mq_m$ admits a finite rational map on an irreducible component of $\mathcal{G}^r_m(\underline{h})$.
Since the component $\mathcal{Z}\in \mathcal{G}^r_m(\underline{h})$ has expected dimension $2g-2+m-\frac{r(r-1)}{2}$ and the point $[C,p_1,\ldots,p_m]\in \mathcal{Z}$ is such that $h_1p_1+\dots+h_mp_m\in C^{(g-1)}\subset \mathcal{Y}\cap \Delta_{\underline{h}}$, there exists an irreducible component $\mathcal{Z}_{\underline{h}}\subset \mathcal{Y}\cap \Delta_{\underline{h}}$ admitting a finite dominant rational map $\mathcal{Z}_{\underline{h}}\dashrightarrow \mathcal{Z}$ given by $h_1q_1+\dots+h_mq_m\in C_t^{(g-1)}\longmapsto [C_t, q_1,\ldots,q_m]$.\\
Let $\mathcal{W}_{\underline{k}}$ be an irreducible component of $\mathcal{Y}\cap \Delta_{\underline{k}}$ containing $\mathcal{Z}_{\underline{h}}$.
By inequality (\ref{equation DIMENSION2}), we have $\dim \mathcal{W}_{\underline{k}}>\dim \mathcal{Z}_{\underline{h}}=2g-3+n-\frac{r(r-1)}{2}$.
Hence the general point $k_1q_1+\dots+k_nq_n\in\mathcal{W}_{\underline{k}}$ is such that the $q_i$'s are distinct.
Thus there exists a finite dominant rational map $\mathcal{W}_{\underline{k}}\dashrightarrow \mathcal{W}\subset \mathcal{M}_{g,n}$ defined as $k_1q_1+\dots+k_nq_n\in C_t^{(g-1)}\longmapsto [C_t, q_1,\ldots,q_n]$, whose image is dense into an irreducible component $\mathcal{W}$ of $\mathcal{G}^r_g(\underline{k})$.
Then, we need to prove that $\mathcal{W}_{\underline{k}}$---and hence $\mathcal{W}$---has the expected dimension, that is $\dim \mathcal{W}_{\underline{k}}=\dim \mathcal{Z}_{\underline{h}}+1$.
As $\mathcal{W}_{\underline{k}}\subset \mathcal{Y}\cap \Delta_{\underline{k}}$ is an irreducible component, $\Delta_{\underline{h}}$ has codimension 1 in $\Delta_{\underline{k}}$, and $\mathcal{Z}_{\underline{h}}$ is an irreducible component of $\mathcal{W}_{\underline{k}}\cap \Delta_{\underline{h}}$, the first part of the assertion follows.\\
In order to conclude the proof, we assume that $[C,p_1,\ldots,p_m]\in \mathcal{Z}$ consists of a non-hyperelliptic curve $C$ such that $h^0(C, \mathcal{O}_C(\sum_{j=1}^m h_j p_j))=r+1$.
Let $U'\subset U^r$ be the image of $\mathcal{W}_{\underline{k}}$ under the morphism ${\mathcal{C}^{(g-1)}\stackrel{\Phi}{\longrightarrow} U^r}$, that is
$\displaystyle U':=\left\{t\in U^r\left|\exists\, q_1,\ldots,q_n\in C_t\textrm{ with } k_1q_1+\dots+k_nq_n\in \mathcal{W}_{\underline{k}}\right.\right\}$, and let ${\Big(\mathcal{C}\stackrel{\phi'}{\longrightarrow} U', \mathcal{L}\longrightarrow\mathcal{C}\Big)}$ be the restriction of the versal deformation family of $\left(C,\mathcal{O}_C(\sum_{j=1}^m h_j p_j)\right)$.
Therefore, for the general point $t\in U'$, there exist $q_1,\ldots,q_n\in C_t$ such that $[C_t,q_1,\ldots,q_n]\in \mathcal{W}$ and $L_t\cong \mathcal{O}_{C_t}\left(\sum_{i=1}^n k_i q_i\right)$.
Moreover, $h_1p_1+\dots+h_mp_m\in\mathcal{Z}_{\underline{h}}\subset \mathcal{W}_{\underline{k}}$, and hence $0\in U'$, i.e. the curve $C=C_0$ is a special fibre of the family $\mathcal{C}\stackrel{\phi'}{\longrightarrow} U'$ with theta-characteristic $L_0\cong \mathcal{O}_C(\sum_{j=1}^m h_j p_j)$.
Thus the general fibre $C_t=\left(\phi'\right)^{-1}(t)$ cannot be hyperelliptic if $C$ is not.
Furthermore, upper semi-continuity of the function $t\longmapsto h^0\left(C_t,L_t\right)$ implies that $h^0\left(C_t, \mathcal{O}_{C_t}\left(\sum_{i=1}^n k_i q_i\right)\right)=r+1$ for general $t\in U'$ (see e.g. \cite[Theorem III.12.8]{Hs}).
\end{proof}
\end{theorem}

\begin{remark}
We would like to note that the very same argument of Theorem \ref{theorem PARTITION} provides an alternative criterion for proving that the $\mathcal{G}^r_g(\underline{k})$'s admit components with expected dimension.
Namely, assume that $\mathcal{Z}\subset \mathcal{M}^r_g$ is an irreducible component having expected dimension, and for some $2\leq m\leq g-1$, let $\underline{k}^{g-1}=(1,\ldots,1)\preceq \underline{k}^{g-2}\preceq\dots\preceq \underline{k}^m$ be a chain of partitions of g-1 such that for all $g-1\geq j\geq m$, the sequence $\underline{k}^j$ consists of $j$ integers.
If for any such a $j$, there exists  $[C,p_1,\ldots,p_j]\in \mathcal{G}^r_g(\underline{k}^j)$ with $\pi_j([C,p_1,\ldots,p_j])\in \mathcal{Z}$, then each $\mathcal{G}^r_g(\underline{k}^j)$ admits a component having expected dimension.
In particular, if $m=g-1$, then the locus $\mathcal{G}^r_g(g-1)$ does.
\end{remark}

\subsection{Loci of subcanonical points}\label{subsection SUBCANONICAL}

Consider a smooth projective curve $C$ of genus $g\geq 2$.
We recall that a point $p\in C$ is a \emph{subcanonical point} if the line bundle $\mathcal{O}_C\left((2g-2)p\right)$ is isomorphic to the canonical bundle $\omega_C$, that is $[C,p]\in \mathcal{G}^r_g(g-1)$ for some $r\geq 0$.
We denote by $\mathcal{G}_g\subset \mathcal{M}_{g,1}$ the locus parameterizing pairs $[C,p]$ such that $p\in C$ is a subcanonical point, and we set hereafter $\mathcal{G}^r_g:=\mathcal{G}^r_g(g-1)$.

When $2\leq g\leq 3$, the loci of subcanonical points have been already described in Example \ref{example LOW GENERA}.
If instead $g\geq 4$, Kontsevich and Zorich \cite{KZ} proved that the locus $\mathcal{G}_{g}$ consists of three irreducible components having dimension  $2g-1$.
The component $\mathcal{G}_g^{\mathrm{hyp}}$ parameterizes hyperelliptic curves endowed with a Weierstrass point (cf. Example \ref{example HYPERELLIPTIC}), whereas the remaining components, $\mathcal{G}_g^{\mathrm{odd}}$ and $\mathcal{G}_g^{\mathrm{even}}$, are described by non-hyperelliptic curves such that $h^0\left(C,\mathcal{O}_C\left((g-1)p\right)\right)$ is odd and even, respectively.
Moreover, Bullock \cite{Bu} showed that the general point $[C,p]\in\mathcal{G}_g^{\mathrm{odd}}$ (resp. $[C,p]\in\mathcal{G}_g^{\mathrm{even}}$) is such that  $h^0\left(C,\mathcal{O}_C\left((g-1)p\right)\right)=1$ (resp. $h^0\left(C,\mathcal{O}_C\left((g-1)p\right)\right)=2$).

By collecting these facts, \cite[Theorems 1.2 and 1.3]{BP} and \cite[Corollary 4.5]{BP}, we obtain the following.
\begin{proposition}\label{proposition LOW DIMENSION}
Let $0\leq r\leq 3$ and let $g(r)$ be defined by (\ref{equation g(r)}).
For any genus $g\geq g(r)$, there exists an irreducible component $\mathcal{Z}\subset\mathcal{G}_{g}^r$ having expected dimension.
Moreover, the general point $[C,p]\in \mathcal{Z}$ satisfies $h^0\left(C,\mathcal{O}_C\left((g-1)p\right)\right)=r+1$ and, apart from the cases $(r,g)=(0,2)$ and $(1,3)$, the curve $C$ is non-hyperelliptic.
\end{proposition}

The following example describes the case $r=2$ with $g=g(r)=6$, and it provides the base case of the induction on $r$ we shall perform in the next section.
\begin{example}[$r=2$ and $g=6$]\label{example G_6^2}
It is well known that the locus $\mathcal{M}^2_6$ is irreducible of pure dimension $12$, and its general point parameterizes a smooth curve $C\subset \mathbb{P}^2$ of degree $5$ (see e.g. \cite[p. 293]{Na}).
Moreover, if $C\subset \mathbb{P}^2$ is such a curve, we have that $\omega_C\cong \mathcal{O}_C(2)$, and $L:=\mathcal{O}_C(1)$ is the unique theta-characteristic on $C$ with $h^0(C,L)=3$.
Therefore, given a plane quintic curve $C\subset \mathbb{P}^2$ and a partition $\underline{k}=(k_1,\ldots,k_n)$ of $g-1=5$, we have $[C,p_1,\ldots,p_n]\in \mathcal{G}^2_6(\underline{k})$ if and only if the divisor $k_1 p_1+\dots +k_n p_n$ is cut out on $C$ by some line $\ell\subset \mathbb{P}^2$.
In particular, it follows from \cite[Section 4.5]{Bu} that the locus of subcanonical points $\mathcal{G}^2_6\subset \mathcal{M}_{6,1}$ admits an irreducible component $\mathcal{Z}_6$ with expected dimension $2g-2=10$, whose general point $[C,p]$ parameterizes a smooth (non-hyperelliptic) quintic curve $C\subset \mathbb{P}^2$ having a 5-fold inflection point at $p\in C$, and $h^0\left(C,O_C\left((g-1)p\right)\right)=3$.
\end{example}

Before stating the last result of this section, we need to recall few basic facts concerning Eisenbud-Harris' theory of limit linear series (see \cite{EH1, EH2}).
Let ${\mathfrak{l}=(L,V)}$ be a $\mathfrak{g}^r_d$ on the curve $C$, that is a line bundle $L$ of degree $d$ endowed with a $(r+1)$-dimensional subspace ${V\subset H^0(C,L)}$.
Given a point $p\in C$, the set $\displaystyle \left\{\ord_p(s)\left|s\in V\right.\right\}$ of orders of vanishing of sections of $V$ at $p$ consists of exactly $r+1$ distinct integers ${0\leq a_0^\mathfrak{l}(p)<a_1^\mathfrak{l}(p)<\dots<a_r^\mathfrak{l}(p)\leq d}$, and the sequence ${\underline{a}^\mathfrak{l}(p):=\left(a_0^\mathfrak{l}(p),a_1^\mathfrak{l}(p),\dots,a_r^\mathfrak{l}(p)\right)}$ is called the \emph{vanishing sequence} of $\mathfrak{l}$ at $p$.
When $\mathfrak{l}$ is the complete canonical series $K_C:=\left(\omega_C, H^0(C, \omega_C)\right)$, the sequence $\left(t_0(p),\dots,t_{g-1}(p)\right)$ such that $t_i(p):=a_i^{K_C}(p)+1$ equals the \emph{sequence of Weierstrass gaps} of $p$, which is the increasing sequence of integers $\displaystyle \left\{0\leq n\leq 2g-1\left|\,h^0\left(C, \mathcal{O}_C((n-1)p)\right)=h^0\left(C, \mathcal{O}_C(n p)\right)\right.\right\}$.
In particular, a point $p\in C$ is subcanonical if and only if $a_{g-1}^\mathfrak{l}(p)=2g-2$.
Furtermore,  $h^0\left(C, \mathcal{O}_C((g-1)p)\right)=r+1$ for some $r\geq 0$, if and only if the set $\displaystyle \left\{i\, \left|\, a_i^\mathfrak{l}(p)\leq g-2\right.\right\}$ consists of exactly $g-1-r$ elements.

So we consider a reducible curve $X=C\cup_q E$ consisting of two smooth curves $C$ and $E$ of genus $g-1$ and $1$ respectively, meeting at a single ordinary node $q$. A \emph{(refined) limit} $\mathfrak{g}^r_d$ on $X$ is
a collection $\mathfrak{l}={\left\{\left(L_{C}, V_{C}\right),\left(L_{E}, V_{E}\right)\right\}}$, where $\mathfrak{l}_{C}:=\left(L_{C}, V_{C}\right)$ and $\mathfrak{l}_{E}:=\left(L_{E}, V_{E}\right)$ are $\mathfrak{g}^r_d$ on $C$ and $E$ respectively, such that the vanishing sequences $\underline{a}^{\mathfrak{l}_{C}}(q)$ and $\underline{a}^{\mathfrak{l}_{E}}(q)$ satisfy the \emph{compatibility conditions} $a_i^{\mathfrak{l}_{C}}(q) + a_{r-i}^{\mathfrak{l}_{E}}(q)= d$ for any $0\leq i\leq r$.

The following theorem follows from \cite[Theorem 4.1 and Corollary 4.4]{BP}, and it shall be involved in the proof of Theorem \ref{theorem EXP DIM}.
\begin{theorem}\label{theorem SMOOTHING g}
Let $r\geq 1$ and assume that there exists an integer $g(r)$ such that $\mathcal{G}^r_{g(r)}$ admits an irreducible component $\mathcal{Z}_{g(r)}$ having expected dimension.
Then for any $g\geq g(r)$, there exists an irreducible component $\mathcal{Z}_{g}$ of $\mathcal{G}^r_{g}$ having expected dimension, as well.\\
Furthermore, if the general point $[C,p]\in\mathcal{Z}_{g(r)}$ satisfies $h^0\left(C,O_C\left((g(r)-1)p\right)\right)=r+1$, then also the general point of $\mathcal{Z}_{g}$ does.
\begin{proof}
The first assertion of the theorem follows straightforwardly from \cite[Theorem 4.1]{BP}.
On the other hand, we want to prove that the general point $[C,p]\in\mathcal{Z}_{g}$ satisfies $h^0\left(C,O_C\left((g-1)p\right)\right)=r+1$.
When $g=g(r)$ the statement is true by assumption. Hence we argue by induction on $g$, and we assume that the assertion holds true up to $g-1$.\\
As in the proof of \cite[Corollary 4.4]{BP}, there exists a family $\left(\mathcal{X}\stackrel{\phi}{\longrightarrow} T, T\stackrel{\rho}{\longrightarrow} \mathcal{X}\right)$ of smooth pointed curves such that $[X_t,p_t]\in \mathcal{Z}_g$ for any $t\neq 0$, and the central fibre $(X_0,p_0)=(C\cup_q E, p)$ is given by a curve $[C,q]\in \mathcal{Z}_{g-1}$ attached at an ordinary node $q$ to an elliptic curve $E$, where $(p-q)\in E$ is a $(2g-2)$-torsion point.
Furthermore, the canonical bundles $\omega_{X_t}\cong \mathcal{O}_{X_t}\left((2g-2)p_t\right)$ induce a limit $\mathfrak{g}^{g-1}_{2g-2}$ on $C\cup_q E$,
$\mathfrak{l}:=\left\{\left(L_{C},V_{C}\right),\left(L_{E},V_{E}\right)\right\}$, where
$L_{C}\cong \omega_C(2q)\cong \mathcal{O}_{C}\left((2g-2)q\right)$ and $L_{E}\cong \mathcal{O}_{E}\left((2g-2)q\right)$.
In particular, since $\mathfrak{l}_C:=\left(L_{C}, V_{C}\right)$ is a base-point-free complete linear series, if $\underline{a}^{K_C}(q):=(a_0,\ldots,a_{g-2})$ is the vanishing sequence of the canonical linear series $K_C$ at the subcanonical point $q\in C$, then $\underline{a}^{\mathfrak{l}_C}(q)=(0,a_0+2,\ldots,a_{g-2}+2)$, where $a_{g-2}+2=2g-2$.
We want to prove that $h^0\left(X_t, \mathcal{O}_{X_t}((g-1)p_t)\right)= r+1$ for general $t\in T$.\\
Compatibility conditions on limit linear series give $\underline{a}^{\mathfrak{l}_E}(q)=(2g-4-a_{g-2},\ldots,2g-4-a_{g-2},2g-2)$.
Moreover, vanishing sequences of $\mathfrak{l}_E$ satisfy $a_i^{\mathfrak{l}_{E}}(p) + a_{g-1-i}^{\mathfrak{l}_{E}}(q)\leq 2g-2$ for any $0\leq i\leq g-1$.
Therefore $\underline{a}^{\mathfrak{l}_E}(p)\leq(0,a_0+2,\ldots,a_{g-2}+2)$ and, by upper semi-continuity of vanishing sequences, we deduce $\underline{a}^{K_{X_t}}(p_t)\leq(0,a_0+2,\ldots,a_{g-2}+2)$ for general $t\in T$.
We recall that $[X_t,p_t]\in \mathcal{G}^r_{g}$ for $t\neq 0$, hence $p_t\in X_t$ is a subcanonical point such that $h^0\left(X_t, \mathcal{O}_{X_t}((g-1)p_t)\right)\geq r+1$ and  $h^0\left(X_t, \mathcal{O}_{X_t}((g-1)p_t)\right)\equiv r+1\,(\mathrm{mod}\,2)$.
Besides, $h^0\left(C,O_C\left((g-2)p\right)\right)=r+1$ by induction, hence the set $\left\{i\, \left|\, a_i\leq g-3\right.\right\}$ consists of exactly $g-2-r$ elements.
Let $\underline{a}^{K_{X_t}}(p_t):=(c_0,\ldots,c_{g-1})$ for general $t\in T$, and let $|\,\cdot\,|$ denote the cardinality of a set.
Then
\begin{displaymath}
\displaystyle \left|\left\{j\, \left|\, c_j\leq g-2\right.\right\}\right|\geq 1+ \left|\left\{i\, \left|\, a_i+2\leq g-2\right.\right\}\right|\geq 1+ \left|\left\{i\, \left|\, a_i\leq g-3\right.\right\}\right|-1=g-1-(r+1).
\end{displaymath}
Thus $r+1\leq h^0\left(X_t, \mathcal{O}_{X_t}((g-1)p_t)\right)\leq r+2$, and the assertion follows as $h^0\left(X_t, \mathcal{O}_{X_t}((g-1)p_t)\right)\equiv r+1\,(\mathrm{mod}\,2)$.
\end{proof}
\end{theorem}

\section{Existence of components with expected dimension}\label{section EXP DIM}

In this section we shall prove Theorem \ref{theorem EXP DIM}.
In the light of Theorem \ref{theorem PARTITION}, Proposition \ref{proposition LOW DIMENSION}, and Theorem \ref{theorem SMOOTHING g}, we essentially need to show that for any $r\geq 4$ and $g(r):= {r+2\choose 2}$, there exists an irreducible component $\mathcal{Z}_{g(r)}$ with expected dimension of the locus $\mathcal{G}^r_{g(r)}$ of subcanonical points, and its general point $[C,p]\in \mathcal{Z}_{g(r)}$ satisfies $h^0\left(C,\mathcal{O}_C\left(\left(g(r)-1\right)p\right)\right)=r+1$.
Large part of the section shall indeed be devoted to prove this fact, which is included in Theorem \ref{theorem SMOOTHING r}.
Actually, the proof shall be mainly set in the Hilbert scheme $\Hilb^r_{g(r),g(r)-1}$ of curves of arithmetic genus $g(r)$ and degree $g(r)-1$ in $\mathbb{P}^r$, with at most nodes as singularities.

Section \ref{subsection COMPONENTS OF M^r_g} shall concern the main results of \cite{Be1}, assuring the existence of an irreducible component $W^r_{g(r)}\subset \Hilb^r_{g(r),g(r)-1}$ whose general point is a smooth curve $C\subset \mathbb{P}^r$ such that $\mathcal{O}_C(1)$ is a large theta-characteristic.
Moreover, we shall slightly improve the description of curves parameterized over $W^r_{g(r)}$.
In Section \ref{subsection LIMIT SUBCANONICAL} we shall extend the notion of 'subcanonical point' to singular curves in $W^r_{g(r)}$, and we shall focus on certain reducible curves admitting such a `limit subcanonical point'.
In Section \ref{subsection SMOOTHING LIMIT SUBCANONICAL} we shall prove that those reducible curves can be deformed to smooth curves endowed with a subcanonical point, and we shall achieve the existence of components $\mathcal{Z}_{g(r)}$ as above.
Finally, Section \ref{subsection PROOF OF EXP DIM} shall be addressed to conclude the proof of Theorem \ref{theorem EXP DIM}.

\subsection{Components of $\mathcal{M}_{g}^r$ having expected dimension}\label{subsection COMPONENTS OF M^r_g}

We summarize and slightly improve some of the results included in \cite{Be1}, where the sharpness of Harris' bound is proved.

For a fixed integer $r\geq 2$, we set hereafter $g=g(r):={r+2\choose 2}$, and we consider the locus $\mathcal{S}^r_{g}$ of pairs $[C,L]$ such that $L$ is a large theta-characteristic on $C$.
Thanks to \cite[Theorem 3.3]{Be1}, there exists an irreducible component $\mathcal{V}_{g}^r\subset \mathcal{S}_{g}^r$ having dimension $3g-3-\frac{r(r+1)}{2}$, whose image under the finite morphism $\varphi\colon \mathcal{S}_g\longrightarrow  \mathcal{M}_g$ is an irreducible component of $\mathcal{M}^r_g$ attaining equality in Harris' bound.
In addition, for a general point $[C,L] \in \mathcal{V}_{g}^r$, the line bundle $L$ is very ample with $h^0(C,L)= r+1$, and the image of the embedding $C\stackrel{\varphi_{|L|}}{\longrightarrow}\mathbb{P}^r$ is a smooth curve---which we still denote by $C\subset \mathbb{P}^r$--- of degree $g-1$ such that the normal bundle $N_{C/\mathbb{P}^r}$ of $C$ in $\mathbb{P}^r$ satisfies $h^1(C,N_{C/\mathbb{P}^r})=0$.

So we consider the Hilbert scheme $\Hilb^{r}_{g,g-1}$ parameterizing reduced curves of arithmetic genus $g$ and degree $g-1$ in $\mathbb{P}^r$ with at most nodes as singularities.
Given a general point $[C,L] \in \mathcal{V}_{g}^r$  and the corresponding embedded curve $C\subset \mathbb{P}^r$, we denote by  $W^r_{g}\subset \Hilb^{r}_{g,g-1}$ the (unique) irreducible component containing the point $[C]$, and by $W^{r}_{g,\text{sm}} \subset W_{g}^r$ the dense open subset parameterizing smooth curves.
The following result is a simple consequence of \cite[Proposition 2.7]{Be2}.
\begin{proposition}\label{proposition W^r_g,sm}
The component $W^r_{g}\subset \Hilb^{r}_{g,g-1}$ has dimension $\dim W^r_{g}=3g-4+{r+2 \choose 2}$, and $[\Gamma,\mathcal{O}_\Gamma(1)]\in \mathcal{S}_{g}^r$ for any smooth curve $[\Gamma]\in W^{r}_{g(r),\text{sm}}$.
\begin{proof}
By \cite[Proposition 2.7]{Be2}, we have that $\dim W^r_{g}=3g-4+{r+2 \choose 2}$, and for general $[C]\in W^{r}_{g,\text{sm}}$, the line bundle $\mathcal{O}_C(1)$ is a very ample theta-characteristic such that $[C,\mathcal{O}_C(1)]\in \mathcal{S}_{g(r)}^r$.
Then, we consider the universal family $\mathcal{C}\subset \mathbb{P}^{r}\times W^{r}_{g,\text{sm}}\xrightarrow{\psi} W^{r}_{g,\text{sm}}$, together with the relative hyperplane bundle $\mathcal{O}_\mathcal{C}(1)$ and the relative quadric bundle $\mathcal{O}_\mathcal{C}(2)$, whose restrictions to the general fibre $C$ of $\psi$ are $\mathcal{O}_C(1)$ and $\mathcal{O}_C(2)$, respectively.
Since $\mathcal{O}_C(1)$ is a theta-characteristic, the canonical bundle on $C$ is $\omega_C\cong \mathcal{O}_C(2)$, and hence $h^1\left(C,\mathcal{O}_C(2)\right)=1$.
By upper semi-continuity of cohomology of $\mathcal{O}_\mathcal{C}(2)$, we deduce that $h^1\left(\Gamma,\mathcal{O}_\Gamma(2)\right)\geq 1$ for any  $[\Gamma]\in W^{r}_{g,\text{sm}}$.
We note that $\deg \mathcal{O}_\Gamma(2)=2g-2$, as $\deg \Gamma=g-1$.
Thus Riemann-Roch theorem assures that  $h^1\left(\Gamma,\mathcal{O}_\Gamma(2)\right)= 1$, and $\mathcal{O}_\Gamma(2)\cong\omega_\Gamma$.
In particular, $\mathcal{O}_\Gamma(1)$ is a theta-characteristic, and the assertion follows.
\end{proof}
\end{proposition}

\begin{remark}\label{remark MODULAR MAP}
In the light of Proposition \ref{proposition W^r_g,sm}, there is a natural modular map $\nu\colon W^{r}_{g,\text{sm}}\longrightarrow \mathcal{V}_{g}^r$ sending $[\Gamma]\in W^{r}_{g,\text{sm}}$ to $[\Gamma,\mathcal{O}_\Gamma(1)]\in \mathcal{S}_{g}^r$.
We point out that, since any smooth curve possesses finitely many theta-characteristics, the dimension of the fibres of $\nu$ equals the dimension of the space $\mathrm{PGL}(r+1)$ of projective transformations of $\mathbb{P}^{r}$, according to the fact that $\dim W^{r}_{g,\text{sm}}=\dim \mathcal{V}_{g}^r+(r+1)^2-1$.
\end{remark}

Turning to singular curves contained in $W^{r}_{g}=W^{r}_{g(r)}$, it follows from \cite[Section 3]{Be1} that for any $r\geq 3$, the component $W^{r}_{g}$ contains the locus $W^{r}_{g,h}$ described by all the reducible nodal curves $X=C \cup E\subset \mathbb{P}^{r}$ such that
\begin{itemize}
  \item[(\emph{i})] $C$ lies into a hyperplane $H\subset \mathbb{P}^{r}$ and, under the isomorphism $H\cong\mathbb{P}^{r-1}$, $[C]\in W^{r-1}_{g(r-1),\text{sm}}$;
  \item[(\emph{ii})] $E$ is an elliptic normal curve of degree $g(r)-g(r-1)=r+1$;
  \item[(\emph{iii})] $C$ and $E$ meet transversally at the $0$-dimensional scheme of length $r+1$ cut out by $H$ on $E$.
\end{itemize}
Furthermore, the locus $W^{r}_{g,h}$ is a divisorial component of $W^r_{g}$. Namely,
\begin{proposition}\label{proposition W0rg}
 For any $r\geq 3$ and $g={r+2 \choose 2}$, the locus $W^{r}_{g,h}$ is equidimensional, and it has codimension 1 in $W^r_{g}$.
\begin{proof}
We want to compute the number $n$ of parameters a general curve $X:=C \cup E$ parameterized by any irreducible component of $W^r_{g,h}$ depends on.
Let $S:=C \cap E$ be the set of singularities of $X$, let $\mathcal{E} \subset \Hilb^{r}_{1,r+1}$ be the subscheme parameterizing smooth elliptic normal curves containing $S$, and let $R \subset \mathcal{E}$ be the irreducible component parameterizing $[E]\in \Hilb^{r}_{1,r+1}$.
By \cite[Lemma 3.1]{Be1}, the normal bundle $N_{E/\mathbb{P}^{r}}$ of $E$ in $\mathbb{P}^r$ satisfies $h^1(E, N_{E/\mathbb{P}^{r}}(-S))=0$.
Thus \cite[Lemma 2.4]{BaBeF} assures that $R$ is smooth at $[E]$. Since, by the same Lemma, the tangent space to $R$ at $[E]$ is isomorphic to $H^{0}(E,N_{E/\mathbb{P}^{r}}(-S))$, one has $\dim R=h^{0}(E,N_{E/\mathbb{P}^{r}}(-S))=h^{0}(E,N_{E/\mathbb{P}^{r}})-(r-1)(r+1)=\dim_{[E]}\text{Hilb}^{r}_{1,r+1}-(r-1)(r+1)=(r+1)^2-(r^2-1)=2r+2$.
As a consequence
\begin{align*}
n    &  =\underbrace{3(g-r-1)-4+{r+1 \choose 2}}_{\dim W^{r-1}_{g-r-1}} + \underbrace{r}_{\textrm{choice of } H \subset \mathbb{P}^{r}} + \underbrace{r+1}_{\textrm{ choice of } S} + \underbrace{2r+2}_{\dim R}=\\
 &  =3g-4+{r+2 \choose 2}-1=\dim W^r_{g}-1.
\end{align*}
\end{proof}
\end{proposition}

\subsection{Nodal curves in $\mathbb{P}^r$ with limit subcanonical points}\label{subsection LIMIT SUBCANONICAL}
Consider the integers $r\geq 2$ and $g=g(r):={r+2\choose 2}$.
As we pointed out in the previous section, the general point of the component $\mathcal{V}_{g}^r\subset \mathcal{S}_{g}^r$ is a pair $[C,L]$, where $[C]\in W^{r}_{g,\text{sm}}$ and $L\cong \mathcal{O}_C(1)$.
We are interested in proving the existence of curves $[C]\in W^{r}_{g,\text{sm}}$ which admit a point $p\in C$ satisfying $L\cong \mathcal{O}_C\left((g-1)p\right)$, so that $p$ is a subcanonical point of $C$ and $[C,p]\in \mathcal{G}^r_g$.
We note that $L\cong \mathcal{O}_C\left((g-1)p\right)$ if and only if there exists a hyperplane $M\subset \mathbb{P}^r$ such that the multiplicity of the intersection between $C$ and $M$ at $p$ is $\mult_p(C,M)=g-1$.
Then it is natural to extend in these terms the notion of subcanonical points to every curve in $W^{r}_{g}$.
Namely,
\begin{definition}
For any curve $[C]\in W^{r}_{g}$, we say that $p\in C$ is a \emph{limit subcanonical point} if there exists a hyperplane $M\subset \mathbb{P}^r$ such that $M$ does not contain any component of $C$, and it cuts out on $C$ a $0$-dimensional scheme of length $g-1$ supported at a $p$.
\end{definition}

We denote by
$$
\displaystyle Q^r_{g}:= \left\{ [C] \in W^r_{g} \, \left| \, \exists\, M\subset \mathbb{P}^r\text{ hyperplane, }\exists\, p \in M \text{ s.t. } \mathrm{mult}_{p}(C,M)=g-1 \right.\right\}
$$
the sublocus of curves admitting a limit subcanonical point, and we set $Q^{r}_{g,\mathrm{sm}}:=Q^r_{g}\cap W^{r}_{g,\mathrm{sm}}$ and $Q^{r}_{g,h}:=Q^r_{g}\cap W^{r}_{g,h}$.
In particular, $Q^r_{g}$ is naturally endowed with the structure of subvariety of $W^r_{g}$.
Analogously, fixing a hyperplane $M\subset \mathbb{P}^r$, we can define the sublocus of $Q^r_{g}(M)\subset Q^{r}_{g}$ of curves possessing a limit subcanonical point cut out by $M$, together with $\displaystyle Q^{r}_{g,\mathrm{sm}}(M):=Q^r_{g}(M)\cap W^{r}_{g,\mathrm{sm}}$ and $\displaystyle Q^r_{g,h}(M):=Q^r_{g}(M)\cap W^{r}_{g,h}$.
\begin{remark}\label{remark Q^r_g(M)}
We recall that any curve $[C]\in Q^r_{g}(M)$ has degree $g-1$, hence there exists a unique limit subcanonical point on $C$ which is cut out by $M$.
Thus, for any hyperplane $M\subset \mathbb{P}^r$, there is a surjective morphism $f\colon Q^r_{g}(M)\longrightarrow M$, whose fibres $Q^r_{g}(M,p):=f^{-1}(p)$ are described by the curves $[C]\in Q^r_{g}(M)$ such that $\mult_p\left(C,M\right)=g-1$.
In addition, for any $p,q\in M$, there are isomorphisms $Q^r_{g}(M,p)\cong Q^r_{g}(M,q)$ induced by projectivities $\tau\in\mathrm{PGL}(r+1)$ such that $\tau(M)=M$ and $\tau(p)=q$.
\end{remark}

We point out that for any hyperplane $M\subset \mathbb{P}^r$, there exists a natural modular map
\begin{eqnarray*}
\mu\colon Q^r_{g,\mathrm{sm}}(M) & \longrightarrow & \mathcal{G}^r_{g}\subset \mathcal{M}_{g,1} \\
\left[C\right] & \longmapsto & \left[C,p\right]
\end{eqnarray*}
where $p\in C$ is the subcanonical point cut out by $M$.
Then the following holds.

\begin{lemma}\label{lemma Q^r_g,sm(M)}
Let $M\subset \mathbb{P}^r$ be a hyperplane, and let $Z\subset Q^r_{g,\mathrm{sm}}(M)$ be an irreducible component, whose general point $[C]\in Z$ parameterizes a non-degenerate linearly normal curve $C\subset \mathbb{P}^r$ such that its normal bundle $N_{C/\mathbb{P}^r}$ satisfies $h^1(C,N_{C/\mathbb{P}^r})=0$.\\
Then $Z$ dominates under $\mu$ an irreducible component $\mathcal{Z}\subset \mathcal{G}^r_{g}$ of dimension $$\dim \mathcal{Z}=\dim Z -\left[(r+1)^2-1-r\right].$$
\begin{proof}
Let $[C]\in Z$ be a general point, and let $p\in C$ be the subcanonical point cut out by $M$, so that $\mu\left([C]\right)=[C,p]$.
Since $h^1(C,N_{C/\mathbb{P}^r})=0$, the point $[C]\in \Hilb^r_{g,g-1}$ is smooth, hence there exists a unique irreducible component of $\Hilb^r_{g,g-1}$ passing through $[C]$, that is $W^r_g$.
Moreover, being $C\subset \mathbb{P}^r$ a smooth linearly normal curve, we deduce that $\mathcal{O}_C(1)\cong\mathcal{O}_C\left((g-1)p\right)$ is a very ample line bundle, with $h^0\left(C,\mathcal{O}_C\left((g-1)p\right)\right)=r+1$.
In particular, if $C^{\prime}\subset \mathbb{P}^r$ is a curve such that $[C^{\prime},p^{\prime}]=[C,p]\in \mathcal{Z}$ for some $p^{\prime}\in C^{\prime}$, then $C^{\prime}$ is projectively equivalent to $C$, so that $[C^{\prime}]\in W^r_g$.\\
Let $\mathcal{Z}\subset \mathcal{G}^r_{g}$ be an irreducible component containing $\mu(Z)$.
Then there exists a non-empty open subset $U\subset\mathcal{Z}$ containing $[C,p]$ such that for any $[\Gamma,q]\in U$, we have that $\mathcal{O}_\Gamma\left((g-1)q\right)$ is a very ample line bundle on $\Gamma$, with $h^0\left(\Gamma,\mathcal{O}_\Gamma\left((g-1)q\right)\right)=r+1$.
In particular, for any $[\Gamma,q]\in U$, the line bundle $\mathcal{O}_\Gamma\left((g-1)q\right)$ gives a complete $\mathfrak{g}^r_{g-1}$ on $\Gamma$.
Then there exists a suitable base change $U^{\prime}\longrightarrow U$, a family of pointed curves $\left(\mathcal{C}\longrightarrow U^{\prime}, \rho\colon U^{\prime}\longrightarrow \mathcal{C}\right)$ and a line bundle $\mathcal{L}$ on $\mathcal{C}$, such that $[C_t,\rho(t)]\in U$ and $L_{|C_t}\cong \mathcal{O}_{C_t}\left((g-1)\rho(t)\right)$ for any $t\in U^{\prime}$.
Moreover, we can define a morphism $\phi\colon\mathcal{C}\longrightarrow \mathbb{P}^{r}$ such that $\phi_{|C_t}\colon C_t\longrightarrow \mathbb{P}^{r}$ is the  complete $\mathfrak{g}^r_{g-1}$ on $C_t$ defined by $L_{|C_t}\cong \mathcal{O}_{C_t}\left((g-1)\rho(t)\right)$.
Furthermore, we may assume that $[C_0,\rho(0)]=[C,p]$ for some $0\in U^{\prime}$, and $\left[\phi(C_0)\right]=[C]\in Z$.\\
Since $[C]\in \Hilb^r_{g,g-1}$ is a smooth point, we necessarily have that $\left[\phi(C_t)\right]\in W^r_g$.
In addition, $\mathcal{O}_{\phi(C_t)}(1)\cong \mathcal{O}_{C_t}\left((g-1)\rho(t)\right)$ for any $t\in U^{\prime}$.
Thus there exists a hyperplane $M_t\subset \mathbb{P}^r$ cutting out on $\phi(C_t)$ the divisor $(g-1)\rho(t)$.
As $Z\subset Q^r_{g,\mathrm{sm}}(M)$ is an irreducible component, and $\phi(C_t)$ specializes to $C=\phi(C_0)$ for general $t\in U^{\prime}$, we conclude that there exists a projective transformation $\tau\in \mathrm{PGL}(r+1)$ such that $\tau(M_t)=M$ and $\left[\tau\left(\phi(C_t)\right)\right]\in Z$.
In particular, $\mu\left(\left[\tau\left(\phi(C_t)\right)\right]\right)=\left[\tau\left(\phi(C_t)\right),\tau(\rho(t))\right]\in \mathcal{Z}$ is a general point, so that $\mu_{|Z}\colon Z\longrightarrow \mathcal{Z}$ is dominant.\\
Finally, since the number of subcanonical points on a smooth curve is finite, the fibre of $\mu$ at some $[C,p]\in \mathrm{Im} (\mu)$ has the same dimension of the space of projectivities in $\mathrm{PGL}(r+1)$ sending $M$ to itself, that is $(r+1)^2-1-r$.
Thus $\dim \mathcal{Z}=\dim \mu(Z)=\dim Z-\left[(r+1)^2-1-r\right]$, as claimed.
\end{proof}
\end{lemma}

\begin{remark}\label{remark Q^r_g,h(M)}
Let $[X]\in Q^r_{g,h}\subset W^r_{g,h}$ be a general point.
Hence $X=C\cup E\subset \mathbb{P}^r$ as in Section \ref{subsection COMPONENTS OF M^r_g}, and there exists a hyperplane $M\subset \mathbb{P}^r$ and a point $p\in M$ such that $\mult_p(X,M)=g(r)-1$.
Since $\deg C=g(r-1)-1$ and $\deg E=g(r)-g(r-1)=r+1$, we have that $p\in C\cap E$, $\mult_p(C,M)=g(r-1)-1$ and $\mult_p(E,M)=r+1$.
In particular, the curve $C\subset H\cong \mathbb{P}^{r-1}$ is such that $[C]\in Q^{r-1}_{g(r-1),\mathrm{sm}}$, and the subcanonical point $p\in C$ is cut out by $H\cap M$.
Besides, the elliptic normal curve $E\subset \mathbb{P}^{r}$ has an inflection point of order $r+1$ at $p$, whose osculating plane is $M$.
\end{remark}

Along the lines of the latter remark, we present a criterion for constructing inductively components of $Q^{r+1}_{g(r+1),h}$ by means of irreducible components of $Q^{r}_{g(r),\mathrm{sm}}$.
\begin{lemma}\label{lemma DIMENSION LIMIT SUBCANONICAL}
Let $r\geq 2$ and $g(r)={r+2\choose 2}$.
Consider a hyperplane $M \subset \mathbb{P}^r$, and let $Z\subset Q^{r}_{g(r),\mathrm{sm}}(M)$ be an irreducible component of dimension $2g(r)-2-\frac{r(r-1)}{2}+(r+1)^2-r$, whose general point $[C]\in Z$ parameterizes a non-degenerate linearly normal curve $C \subset \mathbb{P}^r$ such that $h^1(C,N_{C/\mathbb{P}^r})=0$.\\
Then, for any hyperplane $M^{\prime} \subset \mathbb{P}^{r+1}$, there exists an irreducible component $B \subset Q^{r+1}_{g(r+1),h}(M^{\prime})$ of dimension $2g(r+1)-4-\frac{r(r+1)}{2}+(r+2)^2-r$.
Moreover, the general point $[X] \in B$ parameterizes a non-degenerate linearly normal curve $X\subset \mathbb{P}^{r+1}$ such that $h^1(X,N_{X/\mathbb{P}^{r+1}})=0$.
\begin{proof}
Let $H,M^{\prime} \subset \mathbb{P}^{r+1}$ be distinct hyperplanes and let $p^{\prime} \in H\cap M^{\prime}$.
Let $[C]\in Z$ be as above, and let $p \in C$ be the point such that $\mult_p(C,M)=g(r)-1$.
Let us consider an isomorphism $\iota_{H,p^{\prime}}\colon \mathbb{P}^r\longrightarrow H$ such that $\iota_{H,p^{\prime}}(p)=p^{\prime}$ and the $(r-1)$-plane $M\subset \mathbb{P}^r$ maps to $H\cap M^{\prime}$.
So we identify $p$, $C$ and $M$ with their images in $H\subset \mathbb{P}^{r+1}$.\\
Let $Z_p:=Z \cap Q^r_{g}(M,p)$ be the sublocus of $Z$ described by curves having a subcanonical point at $p\in M$, and notice that it is an irreducible component of $Q^r_{g}(M,p)\cap W^r_{g,\text{sm}}$ having dimension $\dim Z_p=\dim Z-(r-1)$ (cf. Remark \ref{remark Q^r_g(M)}).
Let $C^{(r+2)}$ be the $(r+2)$-fold symmetric product of $C$, and consider the codimension one subvariety $V_{p}:=\left\{\left.p+E\in C^{(r+2)}\right|E\in C^{(r+1)} \right\}$ of effective divisors of degree $r+2$ on $C\subset \mathbb{P}^r$ containing $p$.
We define $R_{V_{p}}\subset \Hilb^{r+1}_{1,r+2}$ as the subscheme parameterizing smooth elliptic normal curves $E$ of degree $r+2$, for which there exists a smooth divisor $S_{p}\in V_{p}$ such that $\Supp(S_{p})\subset E$ and $\mult_{p}(E,M^{\prime})=r+2$. From the fact that $E$ is smooth, it is easy to deduce that the points of $\Supp(S_{p})$ are in linearly general position in $\mathbb{P}^r$.
\begin{claim}\label{claim R_(V_h,p)}
Any irreducible component $R$ of $R_{V_{p}}\subset \Hilb^{r+1}_{1,r+2}$ has dimension $2r+4$.
\begin{proof}[Proof of Claim \ref{claim R_(V_h,p)}]
Fix a general divisor $S_{p}\in V_{p}$, and let $R^{\prime} \subset \Hilb^{r+1}_{1,r+2}$ be the subscheme parameterizing smooth elliptic normal curves $E$ of degree $r+2$, with $\mult_{p}(E,M^{\prime})=r+2$.
We point out that any such a curve may be constructed by taking a smooth elliptic curve $E$ and a point $q\in E$, so that the complete linear system $\displaystyle \left|\mathcal{O}_E((r+2)q)\right|$ defines an embedding $\varphi_{|\mathcal{O}_E((r+2)q)|}\colon E\hookrightarrow \mathbb{P}^{r+1}$, which may be chosen such that $\varphi_{|\mathcal{O}_E((r+2)q)|}(q)=p$ and $\mult_{p}(E,M^{\prime})=r+2$.
Hence $R^{\prime}$ is irreducible, because there is a fibration $R^{\prime}\longrightarrow \mathcal{M}_{1,1}$ over the moduli space $\mathcal{M}_{1,1}$ of pointed elliptic curves given by $\left[E\right]\mapsto [E, p]$, and the fibre is isomorphic to the space $P$ of projectivities of $\mathbb{P}^{r+1}$ fixing $p$ and sending $M^{\prime}$ in itself.
In particular, the dimension of $P$ is $\dim P=\dim \mathrm{PGL}(r+2)-(2r+1)=r^2+2r+2$: it suffices to choose coordinates $x_0,\ldots,x_{r+1}$ on $\mathbb{P}^{r+1}$, set $p=[0: \ldots :0:1]$ and $M^{\prime}=\left\{x_0=0\right\}$, and notice that the space at hand is represented by matrices $\left[a_{i,j}\right]_{0\leq i,j\leq r+1}$ such that $a_{1,l}=0$ for $1\leq l\leq r+1$, and $a_{m,r+1}=0$ for $0\leq m \leq r$.
It follows that $\dim R^{\prime}=\dim \mathcal{M}_{1,1}+ \dim P=r^2+2r+3$.\\
Let $T \subset \mathbb{P}^{r+1}$ be a general hyperplane through $p$ and let $E \cap T=\left\{p,q_1,...,q_{r+1}\right\}$, so that $\left(E \cap T\right)\smallsetminus p$ consists of points of $\mathbb{P}^{r+1}$ in linearly general position.
We note that the set $\Supp(S_{p}) \smallsetminus p\subset \mathbb{P}^{r+1}$ consists of points in linearly general position, as well.
As $\dim P>(r+1)(r+1)$, there exists some projective transformation in $P$ sending $\left(E \cap T\right)\smallsetminus p$ to $\Supp(S_{p}) \smallsetminus p$. Thus there exists a smooth curve $E$ parameterized over $R^{\prime}$ and passing through $\Supp(S_{p})$.
In particular, $E$ intersects the curve $C$ transversally at $S_{p}$.
Moreover, since the points of $\Supp(S_{p}) \smallsetminus p\subset \mathbb{P}^{r+1}$ are projectively equivalent to $r+1$ general points of $\mathbb{P}^{r+1}$, they impose exactly $r(r+1)=r^2+r$ conditions to curves of $R^{\prime}$.
Therefore, if $R_{S_{p}}\subset R^{\prime}$ is the subscheme of curves passing through $\Supp(S_{p})$, and $R^{\prime\prime}\subset R_{S_{p}}$ is an irreducible component containing $\left[E\right]$, we deduce $\dim R^{\prime\prime}=\dim R^{\prime}-\left(r^2+r\right)=r+3$.\\
Finally, we let $S_{p}$ vary in $V_{p}$, and we consider an irreducible component $R$ of $R_{V_{p}}$ containing $R^{\prime\prime}$.
Since $R$ is fibred over $V_{p}$, we conclude that $\dim R=\dim R^{\prime\prime}+\dim V_{p}=2r+4$.
\end{proof}
\end{claim}

We recall that for any $[E]\in R_{V_{p}}$, the elliptic curve $E$ meets $C$ transversally at a general divisor $S_{p}\in V_{p}$, so that the reducible curve $X:=C\cup E$ has arithmetic genus $g(r)+r+2=g(r+1)$.
Furthermore, under the identification of $p\in C$ and $p^{\prime}\in M^{\prime}$, we have $\mult_{p^{\prime}}(X,M^{\prime})=\mult_{p^{\prime}}(C,M^{\prime})+\mult_{p^{\prime}}(E,M^{\prime})=g+r+1=g(r+1)-1$, and hence $\left[X\right]\in Q^{r+1}_{g(r+1),h}(M^{\prime})$.
We shall construct the irreducible component $B\subset Q^{r+1}_{g(r+1),h}(M^{\prime})$, by letting vary $[C]\in Z_p$, the hyperplane $H\subset \mathbb{P}^{r+1}$, $[E]\in R_{V_{p^{\prime}}}$ and $p^{\prime}\in M^{\prime}$.

Let $Z_p^\circ\subset Z_p$ be the open subset parameterizing non-degenerate curves.
Consider the incidence variety $\mathcal{I}:=\left\{\left.\left(p^{\prime},[H]\right)\in M^{\prime}\times \left(\left(\mathbb{P}^{r+1}\right)^*\smallsetminus \left[M^{\prime}\right]\right)\right|p^{\prime}\in H\cap M^{\prime}\right\}$, and let $\iota\colon \mathbb{P}^r\times \mathcal{I}\longrightarrow \mathbb{P}^{r+1}$ be a morphism such that for any pair $\left(p^{\prime},[H]\right)\in \mathcal{I}$, the map $\iota\left(\,\cdot\,,\left(p^{\prime},[H]\right)\right)\colon \mathbb{P}^r\hookrightarrow H$ is an inclusion, with $\iota\left(p,\left(p^{\prime},[H]\right)\right)=p^{\prime}$ and $\iota\left(M,\left(p^{\prime},[H]\right)\right)=H\cap M^{\prime}$.
We point out that $\iota$ induces an inclusion $Z_p^\circ\times \mathcal{I} \hookrightarrow \Hilb^{r+1}_{g,g-1}$, which sends a pair $\left([\Gamma],\left(p^{\prime},[H]\right)\right)$ to the point $\left[\iota\left(\Gamma,\left(p^{\prime},[H]\right)\right)\right]$ parameterizing the curve $\Gamma$ embedded into $H$.
We denote by $Z^{\prime}\subset \Hilb^{r+1}_{g,g-1}$ the image of $Z_p^\circ\times \mathcal{I}$, and let $\alpha_2\colon Z^{\prime}\longrightarrow \mathcal{I}$ be the second projection.
In particular, $Z^{\prime}$ is an irreducible subvariety of dimension $\dim Z_p+\dim \mathcal{I}=\left(\dim Z-(r-1)\right)+2r=\dim Z+r+1$.

We consider the universal family $\mathcal{C}\subset \mathbb{P}^{r+1} \times Z^{\prime} \xrightarrow{\psi} Z^{\prime}$, the relative hyperplane bundle $\mathcal{L}:=\mathcal{O}_{\mathcal{C}}(1)$, and the $(r+2)$-fold relative symmetric product $\mathcal{C}^{(r+2)} \xrightarrow{\Psi} Z^{\prime}$ of $\psi$.
Moreover, we define a divisor $\mathcal{D}\subset \mathcal{C}^{(r+2)}$ as $\mathcal{D}:=\left\{\left.p^{\prime\prime}+D \in \Gamma^{(r+2)} \right| \, [\Gamma]\in Z^{\prime}\textrm{, }\alpha_2[\Gamma]=\left(p^{\prime\prime},[H]\right) \textrm{ and } D\in \Gamma^{(r+1)}\right\}$, whose restriction to any fibre $\Gamma$ of $\psi$ parameterizes the effective divisors on $\Gamma$ having degree $r+2$ and containing the corresponding $p^{\prime\prime}$. Clearly one has $\dim \mathcal{D}=r+1+\dim Z^{\prime}$.
Let $\gamma_2:\mathbb{P}^{r+1} \times \mathcal{D} \rightarrow \mathcal{D}$ be the projection on the second factor, and let $\Hilb_{1,r+2}(\gamma_2)$ be the relative Hilbert scheme of $\gamma_2$ parameterizing curves of arithmetic genus 1 and degree $r+2$ contained in the fibres of $\gamma_2$.
Let $\mathcal{E} \subset \Hilb_{1,r+2}(\gamma_2)$ denote the subscheme described by non-degenerate smooth curves, and let $\varepsilon\colon \mathcal{E}\longrightarrow  \mathcal{D}$ be the natural map inherited from $\Hilb_{1,r+2}(\gamma_2)$.
Then we consider the subvariety of $\mathcal{E}$ given by
$$
\displaystyle \mathcal{Q}:=\left\{\left.[E]\in \mathcal{E}\right| \mult_{p^{\prime\prime}}(E,M^{\prime})=r+2 \textrm{, and }\langle p^{\prime\prime}+D\rangle \cap E=p^{\prime\prime}+D \textrm{ for } \varepsilon\left([E]\right)=p^{\prime\prime}+D \right\},
$$
where $\langle \cdot\rangle$ denotes the linear span in $\mathbb{P}^r$.
Since $V_{p^{\prime}}\subset \mathcal{D}$ by construction, for any irreducible component $R\subset R_{V_{p^{\prime}}}$ as in Claim \ref{claim R_(V_h,p)}, we have that any curve $[E]\in R$ lies on $\mathcal{Q}$.
Conversely, if $[E]\in \mathcal{Q}$ is such that $\varepsilon\left([E]\right)=S_{p^{\prime}}$ with $S_{p^{\prime}}\in V_{p^{\prime}}$ general, then $[E]\in R_{V_{p^{\prime}}}$.
So we fix an irreducible component $R\subset R_{V_{p^{\prime}}}$, and we consider an irreducible component $\mathcal{B}\subset \mathcal{Q}$ containing $R$.
Therefore $\mathcal{B}$ is fibred over the irreducible scheme $Z^{\prime}$ via the composition map
$$
\xi\colon\mathcal{B} \xrightarrow{\varepsilon} \mathcal{D} \xrightarrow{\Psi} Z^{\prime}.
$$
The fibre over a general point $[C^{\prime}] \in Z^{\prime}$, with $\alpha_2[C^{\prime}]=\left(p^{\prime},[H]\right)$, is an irreducible component of the scheme $R_{V_{p^{\prime}}}$, which has dimension $2r+4$ by Claim \ref{claim R_(V_h,p)}.
Thus \begin{align*}
\dim \mathcal{B} & = \dim Z^{\prime} + \dim R_{V_{p^{\prime}}} = \dim Z + r + 1 + 2r+4=\\
 & =\left[ 2g(r)-2-\frac{r(r-1)}{2}+(r+1)^2-r\right]+r+1+ 2r+4=\\
 &  = 2g(r+1)-4-\frac{r(r+1)}{2}+(r+2)^2-r\,.
\end{align*}
Finally, we point out that $\mathcal{B}$ is isomorphic to the scheme
$$
B^\circ:=\left\{\left.[C \cup E] \in W^{r+1}_{g(r+1),h} \right| [C] \in Z^{\prime} \textrm{ and } [E] \in \xi^{-1}([C])\right\}.
$$
We recall that $Z$ is an irreducible component of $Q^{r}_{g,\mathrm{sm}}(M)$, and $\mathcal{D}$ is irreducible.
Then it is straightforward to check that the Zariski closure of $B^\circ$ must describe a whole irreducible component $B$ of $Q^{r+1}_{g(r+1),h}(M^{\prime})$.
\smallskip

For the last part of the statement, let $\left[X\right]\in B$ be a general point, with $X:=C \cup E$, $\alpha_2[C]=\left(p^{\prime},[H]\right)$, and $\varepsilon\left([E]\right)=S_{p^{\prime}}$.
In order to show that $X$ is linearly normal, it suffices to compute the cohomology of the Mayer-Vietoris sequence
$$
0 \rightarrow \mathcal{I}_{X/\mathbb{P}^{r+1}}(1) \rightarrow \mathcal{I}_{C/\mathbb{P}^{r+1}}(1) \oplus \mathcal{I}_{E/\mathbb{P}^{r+1}}(1) \rightarrow \mathcal{I}_{S_{p^{\prime}}/\mathbb{P}^{r+1}}(1) \rightarrow 0.
$$
Since $C \subset H$ is linearly normal by assumption, one has $h^1(\mathcal{I}_{C/H}(1))=0$.
Using the exact sequence $0 \rightarrow \mathcal{I}_{H/\mathbb{P}^{r+1}}(1) \cong \mathcal{O}_{\mathbb{P}^{r+1}} \rightarrow \mathcal{I}_{C/\mathbb{P}^{r+1}}(1) \rightarrow \mathcal{I}_{C/H}(1) \rightarrow 0$, we immediately deduce $h^1(\mathcal{I}_{C/\mathbb{P}^{r+1}}(1))=0$.
On the other hand, $E \subset \mathbb{P}^{r+1}$ is linearly normal, hence $h^1(\mathcal{I}_{E/\mathbb{P}^{r+1}}(1))=0$.
Since $\langle C\rangle =H \cong \mathbb{P}^r$, $\langle E\rangle= \mathbb{P}^{r+1}$ and $\langle S_{p^{\prime}}\rangle=H$, we have $h^0(\mathcal{I}_{C/\mathbb{P}^{r+1}}(1))=1$, $h^0(\mathcal{I}_{E/\mathbb{P}^{r+1}}(1))=0$ and $h^0(\mathcal{I}_{S_{p^{\prime}}/\mathbb{P}^{r+1}}(1))=1$.
Furthermore, $X \subset \mathbb{P}^{r+1}$ is non-degenerate, and hence $h^0(\mathcal{I}_{X/\mathbb{P}^{r+1}}(1))=0$.
Therefore $h^1(\mathcal{I}_{X/\mathbb{P}^{r+1}}(1))=0$, that is $X$ is linearly normal.

To conclude, assume that $h^1(C,N_{C/\mathbb{P}^r})=0$. Since $S_{p^{\prime}} \in V_{p^{\prime}}$ is general, the $r+1$ points of $\Supp(S_{p^{\prime}}) \smallsetminus p^{\prime}$ are general as points of $C$. Thus, it follows immediately from the geometric version of Riemann-Roch theorem that $h^1(C,\mathcal{O}_C(\Sigma + S_{p^{\prime}}))=0$, where $\Sigma$ is a hyperplane section of $C$. Then the proof of the fact that $h^1(X,N_{X/\mathbb{P}^{r+1}})=0$ goes on exactly as in the proof of \cite[Lemma 3.2]{Be1}.
\end{proof}
\end{lemma}

\subsection{Components of $\mathcal{G}_{g(r)}^r$ having expected dimension}\label{subsection SMOOTHING LIMIT SUBCANONICAL}

We are going to prove that any locus $\mathcal{G}^r_{g(r)}\subset \mathcal{M}_{g(r),1}$ of subcanonical points admits an irreducible component $\mathcal{Z}_{g(r)}$ of expected dimension.
In analogy with \cite{Be2}, we shall construct a component $Z_{g(r)}\subset Q^r_{g(r),\mathrm{sm}}$ of suitable dimension in the Hilbert scheme of curves in $\mathbb{P}^r$, and we shall map it in the moduli space $\mathcal{M}_{g(r),1}$ to obtain the component $\mathcal{Z}_{g(r)}\subset \mathcal{G}^r_{g(r)}$.

\begin{theorem}\label{theorem SMOOTHING r}
For any $r\geq 2$ and $g(r)={r+2 \choose 2}$, there exists an irreducible component $\mathcal{Z}_{g(r)}$ of $\mathcal{G}^r_{g(r)}$ having expected dimension $2g(r)-1-\frac{r(r-1)}{2}$, and such that its general point $[C,p]\in\mathcal{Z}_{g(r)}$ satisfies $h^0\left(C,O_C\left(\left(g(r)-1\right)p\right)\right)=r+1$.

\begin{proof}
We want to prove that for any $r\geq 2$ and for any hyperplane $M\subset \mathbb{P}^{r}$, there exists an irreducible component $Z_{g(r)}\subset Q^{r}_{g(r),\mathrm{sm}}(M)$ such that $\dim Z_{g(r)}=2g(r)-2-\frac{r(r-1)}{2}+(r+1)^2-r$, and the general point $[C]\in Z_{g(r)}$ parameterizes a non-degenerate linearly normal curve $C \subset \mathbb{P}^r$ with $h^1(C,N_{C/\mathbb{P}^r})=0$.

Indeed, if such a component exists, Lemma \ref{lemma Q^r_g,sm(M)} assures that the image of $Z_{g(r)}$ under the modular map $\mu\colon Q^{r}_{g(r),\mathrm{sm}}(M)\longrightarrow \mathcal{M}_{g(r),1}$ is an irreducible component $\mathcal{Z}_{g(r)}\subset \mathcal{G}^r_{g(r)}$ having dimension $\dim \mathcal{Z}_{g(r)} = \dim Z_{g(r)}-\left[(r+1)^2-1-r\right]=2g(r)-1-\frac{r(r-1)}{2}$.
Moreover, as the general point $[C]\in Z_{g(r)}$ parameterizes a non-degenerate linearly normal curve $C\subset \mathbb{P}^r$ admitting a divisor $\left(g(r)-1\right)p$ cut out by a hyperplane $M\subset \mathbb{P}^r$, we deduce that $h^0\left(C, \mathcal{O}_C\left(\left(g(r)-1\right)p\right)\right)=h^0\left(C, \mathcal{O}_C\left(1\right)\right)=r+1$, so that the assertion follows.

We point out that when $r=2$, such a component $Z_6=Z_{g(2)}$ does exist.
Given any line $M\subset \mathbb{P}^{2}$, we define $Z_6\subset \Hilb^{2}_{6,5}$ as the locus of smooth quintic curves $C\subset \mathbb{P}^2$ such that $M$ meets $C$ at a single point $p\in C$, as in Example \ref{example G_6^2}.
We recall that pairs $[C,p]\in \mathcal{M}_{6,1}$ as above describe an irreducible component $\mathcal{Z}_6\subset \mathcal{G}_6^2$ having dimension $2g(r)-1-\frac{r(r-1)}{2}=10$.
Therefore, it is easy to see that $\dim Z_{6}=\dim \mathcal{Z}_{6}+(r+1)^2-1-r=16$ (cf. Lemma \ref{lemma Q^r_g,sm(M)}).
We note further that the general curve $[C]\in Z_6$ parameterizes a linearly normal curve, as $\mathcal{O}_C(5p)\cong \mathcal{O}_C(1)$ possesses exactly 3 independent global sections (see Example \ref{example G_6^2}), and that $h^1(C,N_{C/\mathbb{P}^2})=0$ as $N_{C/\mathbb{P}^2} \cong \mathcal{O}_C(C)$.

Then we argue by induction on $r$, and we want to prove that for any hyperplane $M^{\prime}\subset \mathbb{P}^{r+1}$, there exists an irreducible component $Z_{g(r+1)}\subset Q^{r+1}_{g(r+1),\mathrm{sm}}(M^{\prime})$ such that $\dim Z_{g(r+1)}=2g(r+1)-2-\frac{r(r+1)}{2}+(r+2)^2-r-1$, and the general point $[\Gamma]\in Z_{g(r+1)}$ parameterizes a non-degenerate linearly normal curve $\Gamma \subset \mathbb{P}^{r+1}$ with $h^1(\Gamma,N_{\Gamma/\mathbb{P}^{r+1}})=0$.
For the sake of simplicity, we write hereafter $g:=g(r+1)$.

Thanks to the inductive assumption, Lemma \ref{lemma DIMENSION LIMIT SUBCANONICAL} assures that for any hyperplane $M^{\prime}\subset \mathbb{P}^{r+1}$, there exists an irreducible component of $B\subset Q^{r+1}_{g,h}(M^{\prime})$ having dimension $2g-2-\frac{r(r+1)}{2}+(r+2)^2-r-2$.
We focus on the partial compactification of $W^{r+1}_{g,\mathrm{sm}}$ given by $W^{r+1}_{g,\mathrm{sm}}\cup W^{r+1}_{g,h}$.
Then we consider a nodal curve $X_0=C\cup E$ parameterized by a general point of $0=[X_0]\in B$, endowed with a limit subcanonical point $p^{\prime}\in X$ cut out by $M^{\prime}$, i.e. $\mult_{p^{\prime}}(X,M^{\prime})=g-1$.
We are firstly aimed at bounding the dimension of any irreducible component of $Q^{r+1}_{g}(M^{\prime})$ passing through the point $0\in W^{r+1}_{g,\mathrm{sm}}\cup W^{r+1}_{g,h}$.

Let $\mathcal{X}\subset \mathbb{P}^{r+1}\times \left(W^{r+1}_{g,\mathrm{sm}}\cup W^{r+1}_{g,h}\right) \stackrel{\psi}{\longrightarrow} \left(W^{r+1}_{g,\mathrm{sm}}\cup W^{r+1}_{g,h}\right)\subset \Hilb^{r+1}_{g,g-1}$ be the universal family, and let $T\subset W^{r+1}_{g,\mathrm{sm}}\cup W^{r+1}_{g,h}$ be an analytic neighborhood centered at $0$, so that $\dim T=\dim W^{r+1}_{g}=3g-4+{r+3 \choose 2}$ by Proposition \ref{proposition W^r_g,sm}.
Let us still denote by $\mathcal{X}\subset \mathbb{P}^{r+1}\times T \stackrel{\psi}{\longrightarrow} T$ the restriction of the universal family, and let $X_t:=\psi^{-1}(t)$ be the fibre over $t\in T$.
Up to shrinking $T$, we may assume that the hyperplane $M^{\prime}$ has $0$-dimensional intersection with any curve parameterized by $T$, i.e. $M^{\prime}$ does not contain any irreducible component of the fibres $X_t$.
Following the argument for proving Theorem \ref{theorem BOUND}, we consider the $(g-1)$-fold relative symmetric product $\mathcal{X}^{(g-1)}\stackrel{\Psi}{\longrightarrow}  T$ of the family $\psi$, so that the fibre over each $t\in T$ is the variety $\Psi^{-1}(t)=X_t^{(g-1)}$.
In this case $\mathcal{X}^{(g-1)}$ does not coincide with the relative Hilbert scheme $\mathcal{X}^{[g-1]}\stackrel{\widetilde{\Psi}}{\longrightarrow} T$, which parameterizes $0$-dimensional subschemes of $\mathcal{X}$ of length $g-1$ contained in the fibres of $\psi$.
In particular, it follows from \cite[Section 2]{Ra3} that $\mathcal{X}^{[g-1]}$ is a smooth variety of dimension $\dim T + g-1$ fitting in the commutative diagram
\begin{equation*}
\xymatrix{ \mathcal{X}^{[g-1]} \ar[r] \ar[dr]_{\mathfrak{c}} & \mathrm{Bl}_{\Delta}\mathcal{X}^{(g-1)}\ar[d]\\   & \mathcal{X}^{(g-1)}\,, }
\end{equation*}
where $\mathfrak{c}$ is the \emph{cycle map} given by $A\in \mathcal{X}^{[g-1]}\longmapsto\displaystyle \sum_{p\in \mathcal{X}} \mathrm{length}_p(A)\,p\in \mathcal{X}^{(g-1)}$, and $\mathrm{Bl}_{\Delta}\mathcal{X}^{(g-1)}$ is the blow up of $\mathcal{X}^{(g-1)}$ along the discriminant locus $\Delta$ described by non-reduced $0$-cycles (cf. also \cite[Theorem 2.1]{Ra5}).

We define the cycle $\mathcal{Y}\subset\mathcal{X}^{(g-1)}$ parameterizing $0$-cycles cut out by $M^{\prime}$ on the curves $X_t$, that is
\begin{displaymath}
\mathcal{Y}:=\left\{\left.\sum_i a_i p_i\in X_t^{(g-1)}\right| t\in T,\, p_i\in X_t,\, a_i:=\mult_{p_i}(X_t,M^{\prime})\right\}.
\end{displaymath}
So $\mathcal{Y}$ is the image of a section of the map $\mathcal{X}^{(g-1)}\stackrel{\Psi}{\longrightarrow}  T$.
Let $\widetilde{\mathcal{Y}}\subset \mathcal{X}^{[g-1]}$ be the pullback of the strict transform of $\mathcal{Y}$ in $\mathrm{Bl}_{\Delta}\mathcal{X}^{(g-1)}$, and let $\widetilde{\Delta}_{(g-1)}\subset \mathcal{X}^{[g-1]}$ be the locus parameterizing $0$-dimensional schemes supported at a single point.
By \cite[Section 2.1]{Ra4}, the locus $\widetilde{\Delta}_{(g-1)}$ has dimension $\dim T + 1$, and it maps birationally on the relative small diagonal ${\Delta}_{(g-1)}:=\left\{\left.(g-1)q\in X_t^{(g-1)}\right|t\in T\right\}\subset \mathcal{X}^{(g-1)}$.
In particular, the restriction $\mathfrak{c}_{|\widetilde{\Delta}_{(g-1)}}\colon \widetilde{\Delta}_{(g-1)}\longrightarrow \Delta_{(g-1)}$ is an isomorphism away from the nodes of the fibres of $\psi$, whereas if $q\in \Sing(X_t)$ for some $t\in T$, the preimage of $(g-1)q\in X_t^{(g-1)}$ consists of a chain of $g-2$ rational curves (see also \cite[Theorem 1]{Ra2}).

Therefore a point $q\in X_t$ is a limit subcanonical point cut out by $M^{\prime}$---i.e. $t=[X_t]$ lies on $Q^{r+1}_g(M^{\prime},q)$---if and only if $\mathfrak{c}^{-1}\left((g-1)q\right)\in\widetilde{\mathcal{Y}}\cap \widetilde{\Delta}_{(g-1)}$.
Since $\widetilde{\mathcal{Y}}$ maps finitely to its image under $\widetilde{\Psi}$, also $\widetilde{\mathcal{Y}}\cap \widetilde{\Delta}_{(g-1)}$ does.
Moreover, $0\in \widetilde{\Psi} \left(\widetilde{\mathcal{Y}}\cap \widetilde{\Delta}_{(g-1)}\right)\subset T$ as $p^{\prime}\in X_0$ is a limit subcanonical point of $X_0$ cut out by $M^{\prime}$.
Thus any irreducible component $Z^{\prime}\subset Q^{r+1}_{g}\left(M^{\prime}\right)$ passing through $0$ has dimension bounded by
\begin{equation}\label{equation DIM Q^r_g(M')}
\begin{array}{ll}
\dim Z^{\prime} & \geq \dim \widetilde{\mathcal{Y}} + \dim \widetilde{\Delta}_{(g-1)} - \dim \mathcal{X}^{[g-1]} = \dim T + \left(\dim T + 1\right) - \left(\dim T + g - 1\right) =\\
 & \displaystyle= 3g-4+{r+3 \choose 2}-g+2=2g-2- \frac{r(r+1)}{2}+ (r+2)^2-r-1.
\end{array}
\end{equation}
Let us assume that $Z^{\prime}\subset Q^{r+1}_g\left(M^{\prime}\right)\cap \left(W^{r+1}_{g,\mathrm{sm}}\cup W^{r+1}_{g,h}\right)$ is an irreducible component containing $B$, and let $Z_g\subset Z^{\prime}$ be the---possibly empty---sublocus parameterizing smooth curves.
We recall that $B\subset Q^{r+1}_{g,h}\left(M^{\prime}\right)$ is an irreducible component of the intersection $Z^{\prime}\cap W^{r+1}_{g,h}$, and $\dim B=2g-2- \frac{r(r+1)}{2}+ (r+2)^2-r-2$.
By inequality \eqref{equation DIM Q^r_g(M')} we have $\dim Z^{\prime}\geq \dim B+1$, so that $Z_g\subset Q^{r+1}_{g,\mathrm{sm}}\left(M^{\prime}\right)$ is non-empty.
Furthermore, Proposition \ref{proposition W0rg} assures that $W^{r+1}_{g,h}$ is a divisorial component of $W^{r+1}_{g}$, and hence $\dim Z_g=\dim Z^{\prime}= \dim B+1$.

Hence we achieved the existence of an irreducible component $Z_{g(r+1)}=Z_g\subset Q^{r+1}_{g,\mathrm{sm}}\left(M^{\prime}\right)$ having the desired dimension.
It remains to check that its general point $[\Gamma]\in Z_{g}$ parameterizes a linearly normal curve $\Gamma\subset \mathbb{P}^{r+1}$ with $h^1(\Gamma,N_{\Gamma/\mathbb{P}^{r+1}})=0$.
However, these facts follow easily as $B$ lies in the Zariski closure of $Z_{g}\subset \Hilb^{r+1}_{g,g-1}$ , and the general point $[X]\in B$ parameterizes a linearly normal curve $X=C\cup E\subset \mathbb{P}^{r+1}$ with $h^1(X, N_{X/\mathbb{P}^{r+1}})=0$.
\end{proof}
\end{theorem}

\subsection{Proof of Theorem \ref{theorem EXP DIM}}\label{subsection PROOF OF EXP DIM}

In the light of the previous analysis, proving the sharpness of Theorem \ref{theorem BOUND} is now straightforward.

\begin{proof}[Proof of Theorem \ref{theorem EXP DIM}]
As in the statement of the theorem, we consider integers $r\geq 0$ and $g\geq g(r)$, where $g(r)$ is given by (\ref{equation g(r)}), together with a partition $\underline{k}=(k_1,\ldots,k_n)$ of $g-1$.
We are aimed at proving that the locus $\mathcal{G}^r_g(\underline{k})$ admits an irreducible component having expected dimension, whose general point $[C,p_1,\ldots,p_n]$ is such that $h^0\left(C, \mathcal{O}_C\left(\sum^n_{i=1} k_i p_i\right)\right)=r+1$ and---except for the cases $(r,g)=(0,2)$ and $(1,3)$---the curve $C$ is non-hyperelliptic.\\
The case $2\leq g\leq 3$ has been already studied in Example \ref{example LOW GENERA}, and Theorem \ref{theorem EXP DIM} holds under this assumption.
Then we set hereafter $g\geq 4$, and we initially focus on the case $\underline{k}=(g-1)$.\\
In the range $0\leq r\leq 3$, the assertion of Theorem \ref{theorem EXP DIM} on the locus $\mathcal{G}^r_g$ of subcanonical points is covered by Proposition \ref{proposition LOW DIMENSION}.\\
Hence we assume $r\geq 4$, so that $g(r)={r+2\choose 2}$.
By Theorems \ref{theorem SMOOTHING r} and \ref{theorem SMOOTHING g}, for any $g\geq g(r)$ there exists an irreducible component of $\mathcal{Z}_g\subset \mathcal{G}^r_g$ having expected dimension $2g-1-\frac{r(r-1)}{2}$, and such that $h^0\left(C, \mathcal{O}_C\left((g-1) p\right)\right)=r+1$ for general $[C,p]\in \mathcal{Z}_g$.
In particular, the curve $C$ is non-hyperelliptic: we have $\dim \mathcal{Z}_g\leq 2g-7$ for any $r\geq 4$, whereas hyperelliptic curves of genus $g$  just describe an irreducible component $\mathcal{G}_g^{\mathrm{hyp}}\subset \mathcal{G}_g^{r}$ of dimension $2g-1$, with $r\equiv \left\lfloor\frac{g-1}{2}\right\rfloor\,(\mathrm{mod}\,2)$ (see Example \ref{example HYPERELLIPTIC} and Section \ref{subsection SUBCANONICAL}).\\
Therefore Theorem \ref{theorem PARTITION} assures that for any partition $\underline{k}=(k_1,\ldots,k_n)$ of $g-1$, the locus $\mathcal{G}^r_g(\underline{k})$ admits an irreducible component $\mathcal{W}_g$ having expected dimension.
If moreover $[D,q_1,\ldots,q_n]\in \mathcal{W}_g$ is a general point, then neither $D$ is hyperelliptic, nor $h^0\left(D, \mathcal{O}_D\left(\sum^n_{i=1} k_i q_i\right)\right)>r+1$.
\end{proof}

\section*{Acknowledgements}

We would like to thank Enrico Arbarello, Andrea Bruno, Giulio Codogni, Edoardo Sernesi and Filippo Viviani for helpful discussion.

\end{document}